\documentclass[12pt]{article}

\usepackage{amsthm,amsmath,a4wide}
\usepackage{amssymb}
\usepackage[hidelinks]{hyperref}
\input xypic
\usepackage[all,tips]{xy}
\newtheorem{theo}{Theorem}[section]
\newtheorem{lem}[theo]{Lemma}
\newtheorem{prop}[theo]{Proposition}
\newtheorem{corr}[theo]{Corollary}
\newtheorem{conj}[theo]{Conjecture}

\theoremstyle{definition}
\newtheorem{defin}[theo]{Definition}

\theoremstyle{rem}
\newtheorem{rem}[theo]{Remark}
\newtheorem{example}[theo]{Example}

\newcommand{\bth}{\begin{theo}}
\renewcommand{\eth}{\end{theo}}
\newcommand{\bpr}{\begin{prop}}
\newcommand{\epr}{\end{prop}}
\newcommand{\ble}{\begin{lem}}
\newcommand{\ele}{\end{lem}}
\newcommand{\bco}{\begin{corr}}
\newcommand{\eco}{\end{corr}}
\newcommand{\bde}{\begin{defin}}
\newcommand{\ede}{\end{defin}}
\newcommand{\bex}{\begin{example}}
\newcommand{\eex}{\end{example}}
\newcommand{\bre}{\begin{rem}}
\newcommand{\ere}{\end{rem}}
\newcommand{\bcj}{\begin{conj}}
\newcommand{\ecj}{\end{conj}}

\newcommand{\cC}{\mathcal{C}}
\newcommand{\cZ}{\mathcal{Z}}

\newcommand{\D}{\mathcal{D}}

\newcommand{\Rep}{\mathcal{R}{\it ep}}
\newcommand{\Vect}{\mathcal{V}{\it ect}}

\newcommand{\sve}{{\scriptscriptstyle{\vee}}}

\renewcommand{\epsilon}{\varepsilon}

\newcommand{\beq}{\begin{equation}}
\newcommand{\eeq}{\end{equation}}

\newcommand{\ot}{{\otimes}}
\newcommand{\op}{{\oplus}}
\newcommand{\lb}{\label}
\newcommand{\nl}{\newline}
\newcommand{\bpf}{\begin{proof}}
\newcommand{\epf}{\end{proof}}

\newcommand{\C}{{\cal C}}
\newcommand{\Z}{{\cal Z}}
\newcommand{\V}{{\cal V}}
\newcommand{\supp}{\mathrm{supp}}

\everymath{\displaystyle}
\begin{document}

\title{On Lagrangian algebras in group-theoretical braided fusion categories}
\author{Alexei Davydov and Darren Simmons}
\maketitle
\begin{center}
Department of Mathematics, Ohio University, Athens, OH 45701, USA
\end{center}

\begin{abstract}
We describe Lagrangian algebras in twisted Drinfeld centres for finite groups.
\nl
Using the full centre construction, we establish a 1-1 correspondence between Lagrangian algebras and module categories over pointed fusion categories.
\end{abstract}

\tableofcontents

\section{Introduction}

In this paper, we classify Lagrangian algebras in group-theoretical modular categories.
This, in particular, gives a classification of physical modular invariants for group-theoretical modular data (a problem raised in \cite{cg}).
It should be mentioned that the set of labels for physical modular invariants was obtained in \cite{os1} (using the language of module categories). What was missing was a way of recovering the modular invariant corresponding to a label.
By establishing a correspondence between Lagrangian algebras and module categories, and by computing the characters of Lagrangian algebras, we give a method for determining modular invariants corresponding to module categories.

By a group-theoretical modular category $\Z(G,\alpha)$, we mean the monoidal (or Drinfeld) centre $\Z(\V(G,\alpha))$ of the category of vector spaces $\V(G,\alpha)$ graded by a finite group $G$. Here, $\alpha$ is a 3-cocycle of $G$ with coefficients in the multiplicative group $k^*$ of the ground field, which is used to twist the standard associativity constraint for the tensor product of graded vector spaces.  More precisely, we describe objects of $\Z(G,\alpha)$ explicitly as $G$-graded vector spaces with compatible $G$-action (section \ref{gtbc}), and prove later (section \ref{mcgs}) that $\Z(G,\alpha)$ is isomorphic to the monoidal centre $\Z(\V(G,\alpha))$. \newline

A commutative algebra $A$ in a braided fusion category $\C$ is Lagrangian if any local $A$-module in $\C$ is a direct sum of copies of $A$ (we recall basic concepts of braided tensor categories in section \ref{prel}). 
We classify Lagrangian and more general indecomposable commutative separable (etale for short) algebras in $\Z(G,\alpha)$ in two steps.  First, we describe etale algebras with the trivial grading (section \ref{eatd}). These are nothing but indecomposable commutative separable algebras with a $G$-action and hence are just function algebras on transitive $G$-sets (we work over an algebraically closed field $k$).  Up to isomorphism, they are labelled by (conjugacy classes of) subgroups $H\subset G$.  Then we identify the category of local modules $\Z(G,\alpha)_{k(G/H)}^\mathrm{loc}$ with the group-theoretical modular category $\Z(H,\alpha|_H)$ (theorem \ref{comgalg}). A general etale algebra in $\Z(G,\alpha)$ is an extension of its trivial degree component, and hence is an etale algebra in one of $\Z(G,\alpha)_{k(G/H)}^\mathrm{loc}$. Considered as an algebra in $\Z(H,\alpha|_H)$, it has the one-dimensional trivial degree component.  Our second step (section \ref{eatt}) is to classify such algebras (proposition \ref{ttd}) and their categories of local modules (theorem \ref{loc2}).  Then in section \ref{ealm}, we combine the results obtaining the description of all etale algebras in $\Z(G,\alpha)$ (theorem \ref{main}) and their local modules (theorem \ref{local}). As a corollary, we get a classification of Lagrangian algebras in $\Z(G,\alpha)$ (corollary \ref{Lagr}). They are parametrised by (conjugacy classes of) subgroups $H\subset G$ together with a coboundary $d(\gamma)=\alpha|_H$ matching with the answer from \cite{os1}.  In section \ref{fc}, we explain this matching by identifying Lagrangian algebras with full centres of indecomposable separable algebras in $\V(G,\alpha)$ (theorem \ref{fcia}).  Finally, after recalling the character theory for objects of $\Z(G,\alpha)$ (section \ref{char}), we compute characters of Lagrangian algebras (theorem \ref{chla}).  We treat the case $G=S_3$ (the symmetric group on 3 symbols) as an example (section \ref{s3e}).\newline

This paper extends the results of \cite{da3} to the case of $\Z(G,\alpha)$ with a nontrivial cocycle $\alpha$.  It could have been titled ``Modular invariants for group-theoretical modular data II". The scheme of the proof we follow here is very similar to \cite{da3}. However, the present paper is not a mere extension of \cite{da3}:  the presence of a nontrivial cocycle makes all constructions and computations much more elaborate. For example, to identify categories of local modules in section \ref{eatd}, we design a recognition tool by looking at more general fusion categories of group theoretic origin (the appendix).  

\section*{Acknowledgment} 

The paper was completed while the first author was visiting Max Planck Institut f\"ur Mathematik (Bonn, Germany). He thanks the institution for perfect working conditions. 

\section{Preliminaries}\lb{prel}

Here, we recall a number of preliminary concepts.
Throughout $k$ is the ground field of characteristic zero and $\Vect = \Vect_k$ is the category of (finite dimensional) vector spaces over $k$. 

\subsection{Algebras and modules}
Let $\C$ be a monoidal category. 
An {\em associative unital algebra} in $\cC$ is an object $A\in\cC$ together with two morphisms $\mu:A\otimes A\rightarrow A$ and $\iota:I\rightarrow A$ such that 
$$\mu(\mu\otimes \mathrm{Id}_A) = \mu(\mathrm{Id}_A\otimes\mu),$$ and
$$\mu(\iota\otimes \mathrm{Id}_A) = \mathrm{Id}_A =\mu(\mathrm{Id}_A\otimes\iota).$$

By an ``algebra" , we will mean an associative unital algebra.  

Let $A\in\cC$ be an algebra.  A {\em right $A$-module} is an object $M\in\cC$ together with a morphism $\nu:M\otimes A\rightarrow M$ such that $$\nu(\nu\otimes\mathrm{Id}_A) = \nu(\mathrm{Id}_M\otimes\mu).$$  
Left $A$-modules are defined similarly.  
An {\em $A$-$B$-bimodule} $M\in\cC$ is a left $A$-module and a right $B$-module ($B\in\C$ is another algebra) such that the diagram of module action maps
$$\xymatrix{A\ot M\ot B\ar[r] \ar[d] & A\ot M \ar[d]\\ M\ot B\ar[r] & M}$$
commutes.  
\nl
We denote the category of right $A$-modules by $\cC_A$, and that of $A$-$B$-bimodules by $_{A}\cC_{B}$.

\subsection{Separable algebras}

For an algebra $A$ in a spherical category (with the spherical structure $s_X:X\to X^{**}$) denote by $\tau:A\to I$ (the {\em canonical trace function}) the composition
$$\xymatrix{A \ar[r]^(.35){1coev_A} & A\ot A\ot A^* \ar[r]^(.55){\mu 1} & A\ot A^* \ar[r]^{s_A1} & A^{**}\ot A^* \ar[r]^(.65){ev_{A^*}} & I}$$
An algebra $A$ is said to be {\em separable} if the following composition is a nondegenerate pairing (denoted $e:A\otimes A\to I$):  
$$\xymatrix{A\ot A \ar[r]^{\mu} & A \ar[r]^{t_A} & I}$$ 
Nondegeneracy of $e$ means that there is a morphism $\kappa:I\to A\otimes A$ such that the compositions \[A\overset{1\otimes\kappa}{\longrightarrow}A^{\otimes3}\overset{e\otimes1}{\longrightarrow}A\]
\[A\overset{\kappa\otimes1}{\longrightarrow}A^{\otimes3}\overset{1\otimes e}{\longrightarrow}A\] are the identity.  

If $A\in\cC$ is a separable algebra in a spherical fusion category, then the category $\cC_A$ of right $A$-modules is semisimple.

We call an algebra $A$ in a monoidal category $\C$ {\em simple} if any algebra homomorphism $A\to B$ is a monomorphism. 
\ble\lb{is}
Let $A$ be an indecomposable separable algebra in a spherical fusion category $\C$.
Then $A$ is simple.
\ele
\bpf
If an algebra $A$ is not simple then there is a surjective (but not bijective) algebra homomorphism $A\to B$. 
Via the inverse image along this homomorphism the category of $B$-modules $\C_B$ becomes a full subcategory of $\C_A$.
Moreover $\C_B$ is a full left $\C$-module subcategory of $\C_A$.
\nl
Recall from \cite[Proposition 3.9.]{eo} that a semisimple module category over a fusion category is a direct sum of its simple module subcategories.
In particular, $\C_B$ is a direct summand of $\C_A$ as a $\C$-module category.
Hence the algebra $\C(I,A)$ in $\Vect$ (which coincides with the algebra $\mathrm{End}_\C\left(\mathrm{Id}_{\C_A}\right)$ of $\C$-module endomorphisms of the identity functor of $\C_A$) is a non-trivial direct sum.
Thus $A$ is decomposable.
\epf

\subsection{Commutative algebras and local modules}

Let now $\cC$ be a braided monoidal category with the braiding $c_{X,Y}:X\otimes Y\overset{\sim}{\longrightarrow}Y\otimes X.$  
\nl
An algebra $A=\left\langle A,\mu,\iota\right\rangle$ in $\C$ is {\em commutative} if $$\mu=\mu\circ c_{A,A}.$$  
Following \cite{dmno} we call an indecomposable commutative and separable algebra {\em etale}.

A right $A$-module $M=\left\langle M,\nu\right\rangle$ is said to be {\em local} if the following diagram commutes:
$$\xymatrix{M\otimes A  \ar[r]^\nu \ar[d]_{c_{M,A}} & M\\ A\otimes M \ar[r]^{c_{A,M}} & M\otimes A \ar[u]_\nu}$$
We denote by $\C^\mathrm{loc}_A$ the full subcategory of $\C_A$ consisting of local right $A$-modules.  
\nl
Recall from \cite{pa} that (in case $\C$ has coequalisers) the category $\C_A$ is monoidal with respect to the relative tensor product $\ot_A$
and that the category $\C^\mathrm{loc}_A$ is braided. 

An algebra $L$ in a(n additive) braided monoidal category $\cC$ is said to be a {\em Lagrangian algebra} if any local $L$-module is a direct sum of copies of the regular module $L$.
\nl
Note that in the case when $\C$ is a $k$-linear braided monoidal category $L$ is Lagrangian iff the category $\C^\mathrm{loc}_L$ is equivalent to the category $\Vect_k$ of $k$-vector spaces.\newline

The following statement was proved in \cite{ffrs}; we state it here without proof.
\begin{lem}\label{algloc}
Let $\left(A,m,i\right)$ be a commutative algebra in a braided category $\cC$. Let $\left(B,\mu,\iota\right)$ be an algebra in ${\cC}_{A}$.  Define $\overline\mu$ and $\overline\iota$ as compositions
$$\xymatrix{B\otimes B \ar[r] & B\otimes_AB \ar[r]^<<<<{\mu} & B,& & & 1 \ar[r]^i & A \ar[r]^\iota & B.}$$ Then $(B,\overline\mu,\overline\iota)$ is an algebra in $\cC$.\newline
\noindent The map $\iota:A\to B$ is a homomorphism of algebras in $\cC$.\newline
\noindent The algebra $\left(B,\overline\mu,\overline\iota\right)$ in $\cC$ is separable or commutative if and only if the algebra $\left(B,\mu,\iota\right)$ in $\cC_A$ is such.\newline
The functor $\left(\cC_{A}^\mathrm{loc}\right)_{B}^\mathrm{loc}\longrightarrow\cC_{B}^\mathrm{loc}$ given by
\begin{equation}\label{locloc}
\left(M,m:B\otimes_A M\to M\right)\mapsto \left(M,\overline m:B\otimes M\to B\otimes_A M\stackrel{m}{\longrightarrow} M\right)
\end{equation}
\noindent is a braided monoidal equivalence.
\end{lem}

\subsection{Full centre}

Recall from \cite{da4} that the {\em full centre} $Z(A)$ of an algebra $A$ in a monoidal category $\C$ is an object of the monoidal centre $\cZ(\cC)$ together with a morphism $Z(A)\to A$ in $\cC$, terminal among pairs $(Z,\zeta)$, where $Z\in\cZ(\cC)$ and $\zeta:Z\to A$ is a morphism in $\cC$ such that the following diagram commutes:
\begin{equation}\label{cp}
\xymatrix{ A\otimes Z \ar[r]^{A\zeta} \ar[dd]_{z_A} & A\otimes A \ar[rd]^\mu\\ & & A\\ 
 Z\otimes A \ar[r]_{\zeta A} 
 & A\otimes A \ar[ur]_\mu }
\end{equation}
Here $z_A$ is the half-braiding of $Z$ as an object of $\cZ(\cC)$.  The terminality condition means that for any such pair $(Z,\zeta)$ there is a unique morphism $Z\to Z(A)$ in the monoidal centre $\cZ(\cC)$, which makes the diagram
$$\xymatrix{ Z \ar[rr] \ar[rd]_\zeta && Z(A) \ar[ld] \\ & A & }$$ commute. 

Recall that the {\em $\alpha$-induction} is the tensor functor 
$$\alpha:\Z(\C)\to {_A{\C_A}},\qquad \alpha(Z) = Z\ot A,$$
with the left $A$-module structure given by 
$$\xymatrix{A\ot Z\ot A \ar[r]^{z_A\otimes1} & Z\ot A\ot A \ar[r]^{1\otimes\mu} & Z\ot A}$$

\bpr\lb{mfz}
The full centre $Z(A)$ of an indecomposable separable algebra $A$ in a fusion category $\C$ coincides with
the action internal end $[A,A]\in\Z(\C)$ of the trivial bimodule $A\in {_A{\C_A}}$ with respect to the $\Z(\C)$-action on ${_A{\C_A}}$ given by $\alpha$-induction.
\nl
The category of modules $\Z(\C)_{Z(A)}$ is equivalent, as a fusion category, to the category ${_A{\C_A}}$ of $A$-bimodules. 
\epr
\bpf
The universal property of the action internal end says that $[A,A]$ is the terminal object among $(Z,z)$ where $\xi\in\Z(\C)$ and $\xi:\alpha(Z)\to A$ is a morphism of $A$-bimodules. The right $A$-module map $\xi:Z\ot A\to A$ is completely determined by the morphism $\zeta=\xi(1\ot\iota):Z\to A$ (which is still a left $A$-module map). The left $A$-module property for $\zeta$ is exactly \eqref{cp}.
\nl
According to \cite{os} the functors
$$\xymatrix{\Z(\C)_{[A,A]} \ar@/^15pt/[rr]^{-\ot_{[A,A]}A} && {_A{\C_A}} \ar@/^15pt/[ll]^{[A,-]} }$$
are quasi-inverse equivalences.
The tensor structure
\beq\lb{msih}[A,M]\ot_{[A,A]}[A,N]\ \to\ [A,M\ot_AN]\eeq for the functor $[A,-]$ comes from the universal property of the action internal hom.
Indeed the composition 
$$\xymatrix{\alpha([A,M]\ot [A,N])\ot_A A \simeq \alpha([A,M])\ot_A A\ot_A\alpha([A,N])\ot_A A \ar[r] & M\ot_A N}$$ 
induces the map $[A,M]\ot [A,N]\to [A,M\ot_AN]$, which is naturally $[A,A]$-balanced, i.e. factors through $[A,M]\ot_{[A,A]}[A,N]$ giving rise to \eqref{msih}.
\epf
\bre
Here is a slightly different proof of the second statement of proposition \ref{mfz}. 
The canonical braided equivalence (Morita invariance of the monoidal centre) $\Z(\C)\to \Z({_A{\C_A}})$ sends the full centre $Z(A)$ to the full centre $Z(I)$ of the monoidal unit $I\in {_A{\C_A}}$ (which is really the $A$-bimodule $A$). 
\nl
For a fusion category $\D\left(={_A{\C_A}}\right)$, the full centre $Z(I)\in\Z(\D)$ coincides with $L(I)$, where $L:\D\to \Z(\D)$ is the adjoint to the forgetful functor $F:\Z(\D)\to \D$. The adjunction is monadic.  Moreover, the monad $T = L\circ F$ on $\Z(\D)$ is a $\Z(\D)$-module functor. Thus $T$-algebras are the same as $T(I)$-modules. Finally, $T(I) = L(I) = Z(I)$ and the forgetful functor factorises:
$$\xymatrix{\Z(\D) \ar[d] \ar[rd]^F \\ \Z(\D)_{Z(I)} \ar[r]^(.6){\sim} & \D}$$
\ere
\bth
The full centre $Z(A)$ of an indecomposable separable algebra $A$ in a fusion category $\C$ is a Lagrangian algebra in $\Z(\C)$.
\eth
\bpf
The tensor equivalence $\Z(\C)_{Z(A)} \to {_A{\C_A}}$ from proposition \ref{mfz} induces a braided tensor equivalence $\Z(\Z(\C)_{Z(A)}) \to \Z({_A{\C_A}})$.  By Morita invariance of the monoidal centre, $\Z({_A{\C_A}}) \simeq \Z(\C)$.
\nl
By \cite[Proposition 3.7]{dgno}, we have the decomposition into Deligne product $\Z(\C)\boxtimes\overline{\Z(\C)_{Z(A)}^\mathrm{loc}}\simeq \Z(\Z(\C)_{Z(A)})$.
\nl
Combining the above, we get $\Z(\C)\boxtimes\overline{\Z(\C)_{Z(A)}^\mathrm{loc}}\simeq\Z(\C)$, which means that $\Z(\C)_{Z(A)}^\mathrm{loc}\simeq\Vect$, i.e. $Z(A)$ is Lagrangian.
\epf

\section{Commutative algebras in group-theoretical categories}\label{alggrth}

\subsection{Group-theoretical braided fusion categories}\lb{gtbc}

Let $G$ be a finite group. By $k^*$ we denote the multiplicative group of the ground field $k$.  
By a {\em 3-cocycle} of $G$ with coefficients in $k^*$, we mean a normalised 3-cocycle of the standard complex, {\em i.e.}, a function $\alpha:G\times G\times G\rightarrow k^*$ such that 
$$\alpha(g,h,l)\alpha(f,gh,l)\alpha(f,g,h)=\alpha(fg,h,l)\alpha(f,g,hl),\qquad f,g,h,l\in G$$
and such that $\alpha(f,g,h)=1$ each time one of $f,g,h$ is the identity.
\nl
We denote by $Z^3(G,k^*)$ the group of normalised 3-cocycles.

A vector space $V$ is {\em $G$-graded} if there given a direct sum decomposition $V=\op_{g\in G}{V_g}$.  The tensor product of graded vector spaces is graded in a natural way $(V\ot U)_f = \op_{gh=f}V_g\ot U_h$. The monoidal unit in $\mathcal{V}\left(G,\alpha\right)$ is $I=I_e=k$. 
A 3-cocycle $\alpha\in Z^3(G,k^*)$ can be used to twist the standard associativity constraint:
\beq\lb{ass}
\alpha_{U,V,W}\left(u\otimes\left(v\otimes w\right)\right) = \alpha(f,g,h)\left(u\otimes v\right)\otimes w,\qquad u\in U_f, v\in V_g, w\in W_h\ .
\eeq
Denote by $\V(G,\alpha)$ the category of $G$-graded vector spaces with grading preserving linear maps, equipped with the above structure of a fusion category,

An {\em $\alpha$-projective $G$-action} on a $G$-graded vector space $V$ is a collection of automorphisms 
$$f:V\to V,\qquad v\mapsto f.v$$ for each $f\in G$ such that $f\left(V_g\right) = V_{fgf^{-1}}$, and
\begin{equation}\label{pa}
(fg).v = \alpha(f,g|h)f.\left(g.v\right),\quad \forall v\in V_h.
\end{equation}
Here,
\begin{equation}\label{tefa}
\alpha(f,g|h) = \alpha(f,g,h)^{-1}\alpha(f,{^{g}{h}},g)\alpha({^{fg}{h}},f,g)^{-1}.
\end{equation}where $^{f}{h} = fhf^{-1}$. Similarly, define
\begin{equation}\label{acfa}
\alpha(f|g,h) = \alpha(f,g,h)\alpha({^{f}{g}},f,h)^{-1}\alpha({^{f}{g}},{^{f}{h}},f),
\end{equation}
The following identities follow directly from the 3-cocycle condition for normalized $\alpha$:
\beq\lb{pca}\alpha(f,gh|u)\alpha(g,h|u) = \alpha(fg,h|u)\alpha(f,g|huh^{-1}),\eeq
$$\alpha(fg|u,v)\alpha(f,g|u)\alpha(f,g|v) = \alpha(f,g|uv)\alpha(g|u,v)\alpha(f|gug^{-1},gvg^{-1}),$$
$$\alpha(g,h,u)\alpha(f|gh,u)\alpha(f|g,h) = \alpha(f|g,hu)\alpha(f|h,u)\alpha(fgf^{-1},fhf^{-1},fuf^{-1}).$$\newline

\noindent Define the category $\Z(G,\alpha)$ as follows.  Objects of $\Z(G,\alpha)$ are $G$-graded vector spaces  together with $\alpha$-projective $G$-action.  Morphisms are grading and action-preserving homomorphisms of vector spaces. 
\nl
Define the tensor product in $\Z(G,\alpha)$ is the tensor product of $G$-graded vector spaces, with $\alpha$-projective $G$-action defined by
\begin{equation}\label{tp}
f.\left(x\otimes y\right) = \alpha(f|g,h)\big(f.x\otimes f.y\big),\quad x\in X_g,\ y\in Y_h.
\end{equation}
The associativity is given by \eqref{ass}.
\nl
The monoidal unit is $I=I_e=k$ with trivial $G$-action.
\newline
The braiding is given by
\begin{equation}\label{br}
c_{X,Y}(x\otimes y) = f.y\otimes x,\quad x\in X_f, y\in Y.
\end{equation}

The following is well-known. We leave the proof to the reader (see also proposition \ref{mc}).
\bpr
The category $\Z(G,\alpha)$ defined above is braided monoidal.
\epr

\ble\lb{do}
The category $\Z(G,\alpha)$ is rigid, with dual objects $X^\sve = \oplus_f (X^\sve)_f$ given by
$$(X^\sve)_f = (X_{f^{-1}})^\sve = \mathrm{Hom}_{k}(X_{f^{-1}},k),$$ with the $\alpha$-projective action
$$(g.l)(x) = \frac{\alpha(g^{-1},g|f^{-1})}{\alpha(g|f,f^{-1})}l((g^{-1}.x)),\qquad l\in \mathrm{Hom}_{k}(X_{f^{-1}},k),\quad x\in X_{gf^{-1}g^{-1}}\ .$$
\ele
\bpf
It can be verified that the formula above indeed defines an $\alpha$-projective action on the graded vector space.
The evaluation morphism $X^\sve\ot X\to I$ is given by the canonical pairing 
$(X^\sve)_f \ot X_{f^{-1}}\to k$. Similarly the coevaluation morphism $I\to X\ot X^\sve$ is given by the canonical elements 
$k\to X_{f^{-1}}\ot (X^\sve)_f$. 
\epf

\ble\label{simpobj}
The category $\Z(G,\alpha)$ admits the following decomposition into a direct sum of $k$-linear subcategories:
\begin{equation}
\Z(G,\alpha)=\bigoplus_{f\in cl(G)}\Z_f(G,\alpha),
\end{equation}
where the sum is taken over a set $cl(G)$ of representatives of conjugacy classes of elements of $G$, and for $f\in G$, the subcategory $\Z_f(G,\alpha)$ is given by $$\Z_f(G,\alpha)=\left\{Z\in\Z(G,\alpha)\Big|\  \supp(Z)=f^G\right\}.$$  Here, $f^G=\left\{gfg^{-1}|g\in G\right\}$ denotes the conjugacy class of $f$ in $G$.
\ele
\begin{proof}
Clearly, the support of an object of $\Z(G,\alpha)$ is a union of conjugacy classes in $G$. It is also straightforward that for $Z\in \Z(G,\alpha)$, one has $$Z = \op_{c\in  cl(G)} Z_c,\qquad Z_c = \op_{f\in c}Z_f$$ is a decomposition into a direct sum of objects in $\Z(G,\alpha)$. \end{proof}

\ble\label{cat1}
The category $\Z_f(G,\alpha)$ is equivalent, as a $k$-linear category, to the category $k\left[C_{G}\left(f\right),\alpha\left(\ ,\ |f\right)^{-1}\right]$-$\mathrm{Mod}$ of left modules over the twisted group algebra $k\left[C_{G}\left(f\right),\alpha\left(\ ,\ |f\right)^{-1}\right]$.
\ele
\bpf
We show that the functor
$$F:\Z_f(G,\alpha)\rightarrow k\left[C_{G}\left(f\right),\alpha\left(\ ,\ |f\right)^{-1}\right]-\mathrm{Mod},\qquad F(Z)=Z_f$$ 
is an equivalence by exhibiting its quasi-inverse $E:k\left[C_{G}\left(f\right),\alpha\left(\ ,\ |f\right)^{-1}\right]$-$\mathrm{Mod}\rightarrow \Z_f(G,\alpha)$ given by
$$E(M)=\left\{a:G\rightarrow M\Big|\ a\left(xy\right)=\alpha\left(y^{-1},x^{-1}|xfx^{-1}\right)y^{-1}.a(x)\qquad\forall x\in G, y\in C_{G}\left(f\right)\right\}.$$
The $G$-grading on $E(M)$ is defined as follows:  a function $a\in E(M)$ is homogeneous if and only if $\supp(a)$ is a single coset modulo $C_{G}\left(f\right)$.  More precisely,
$$
|a|=xfx^{-1}\ \iff\ \supp(a)=xC_{G}(f).
$$
The $\alpha$-projective $G$-action on $E(M)$ is given by 
$$\left(g.a\right)(x)=\alpha\left(x^{-1},g|xfx^{-1}\right)a\left(g^{-1}x\right).$$
It is straightforward that these definitions make $E(M)$ into an object of $\Z_f(G,\alpha)$.
\nl
Now we define the adjunction isomorphisms $\phi:Id\to E\circ F$ and $\psi:F\circ E\to Id$. 
For $Z\in\Z_f(G,\alpha)$ and a homogeneous $z\in Z_{xfx^{-1}}$ define $\tilde{z}:G\rightarrow Z_f$ by 
$$
\tilde{z}(g)=\left\{\begin{array}{cc}g^{-1}.z, & g\in xC_{G}(f)\\ 0, & \mathrm{otherwise}\\ \end{array}\right..
$$
It is straightforward to see that $\tilde{z}\in E\left(F(Z)\right)$.  
\nl
For an object $V\in\mathcal{C}\left(\{f\},C_{G}(f),\alpha\right)$, define $\psi_V:F\left(E(V)\right)\rightarrow V$ by 
$\psi_V(a)=a(e)$.  One can verify directly that $\phi$ and $\psi$ are morphisms and inverse to one another.
\epf

\bco\label{simpl}
Simple objects $Z\in\Z(G,\alpha)$ are labelled by conjugacy classes of pairs $(f,M)$ where $f\in G$ and $M$ is a simple $k\left[C_{G}\left(f\right),\alpha\left(\ ,\ |f\right)^{-1}\right]$-module.
\eco

\bco
Let $k$ be a field of characteristic zero.
The category $\Z(G,\alpha)$ is fusion.
\eco

\subsection{Algebras in group-theoretical modular categories}

We start with expanding the structure of an algebra in the category $\Z(G,\alpha)$ in plain algebraic terms. Recall that a $G$-graded vector space $A = \op_{g\in G}A_g$ is a {\em $G$-graded algebra} if it is equipped with a grading preserving multiplication  $A_fA_g\subseteq A_{fg}$.  
We call a $G$-graded algebra $A$ {\em $\alpha$-associative} (for a 3-cocycle $\alpha\in Z^3(G,k^*)$) if $$a(bc) = \alpha(f,g,h)(ab)c,\qquad \forall a\in A_f, b\in A_g, c\in A_h.$$

\begin{prop}
An algebra $A$ in the category $\Z(G,\alpha)$ is a $G$-graded $\alpha$-associative algebra together with an $\alpha$-projective $G$-action such that
\begin{equation}\label{ah}
f.(ab) = \alpha(f|g,h)\left(f.a\right)\left(f.b\right),\quad a\in A_g, b\in A_h.
\end{equation}
An algebra $A$ in the category $\Z(G,\alpha)$ is commutative iff 
\begin{equation}\label{co}
ab = \left(f.b\right)a,\quad \forall a\in A_f, b\in A.
\end{equation}
\end{prop}
\begin{proof}
Being a morphism in the category $\Z(G,\alpha)$, the multiplication of an algebra in $\Z(G,\alpha)$ preserves grading and $\alpha$-projective $G$-action (hence the property (\ref{ah})). Associativity of multiplication in $\Z(G,\alpha)$ is equivalent to $\alpha$-associativity.
\newline
The formula (\ref{br}) for the braiding in $\Z(G,\alpha)$ implies that commutativity for an algebra $A$ in the
category $\Z(G,\alpha)$ is equivalent to the condition (\ref{co}).
\end{proof}

A {\em $G$-algebra} is an algebra $A$ (in $\Vect$) together with an action of $G$ on $A$ by algebra homomorphisms.
\begin{corr}
The degree $e$ part $A_e$ of an algebra $A$ in the category $\Z(G,\alpha)$ is an associative $G$-algebra and $A$ is a
module over $A_e$. The algebra $A_e$ is commutative if $A$ is a commutative algebra in the category $\Z(G,\alpha)$.
\end{corr}
\begin{proof}
The normalization condition for the cocycle $\alpha$, together with $\alpha$-associativity of $A$, implies that $A_e$ is
an associative algebra and that $A$ is a $A_e$-module. The same normalization condition, together with $\alpha$-projectivity
of the $G$-action and the property (\ref{ah}), implies that the action of $G$ on $A_e$ is genuine and that $G$ acts on $A_e$
by algebra homomorphisms. Commutativity of $A_e$ for a commutative algebra $A\in \Z(G,\alpha)$ follows directly from
the identity (\ref{co}).
\end{proof}

\begin{prop}\label{sep}
An algebra $A$ in the category $\Z(G,\alpha)$ is separable if and only if the composition \[A_f\otimes A_{f^{-1}}\overset{\mu}{\longrightarrow}A_e\overset{\tau}{\longrightarrow}k\] defines a nondegenerate bilinear pairing for any $f\in G$. In particular, the algebra $A_e$ is separable if $A$ is a separable algebra in the category $\Z(G,\alpha)$.
\end{prop}
\begin{proof}
Being a homomorphism of graded vector spaces, the standard trace map $\tau: A\to I$ is zero on $A_f$ for $f\not= e$. Hence the standard bilinear form is zero on $A_f\otimes A_g$ unless $fg=e$. In particular, the restriction of $\tau$ to $A_e$ makes it a separable algebra in the category of vector spaces.
\end{proof}

\subsection{Etale algebras in trivial degree and their modules}\lb{eatd}

We start with a well known (see for example \cite{ko}) description of etale $G$-algebras in $\Vect$. We give (a sketch of) the proof for completeness.
\begin{lem}\label{comgalg}
Etale $G$-algebras are function algebras on $G$-sets. Indecomposable $G$-algebras correspond to transitive $G$-sets.
\end{lem}
\begin{proof}
An etale algebra over an algebraically closed field $k$ is a function algebra $k(X)$ on a finite set $X$ (with elements of $X$ corresponding to minimal idempotents of the algebra). The $G$-action on the algebra amounts to a $G$-action on the set $X$. Obviously, the algebra of functions $k(X\cup Y)$ on the disjoint union of $G$-sets is the direct sum of $G$-algebras $k(X)\oplus k(Y)$ and any direct sum decomposition of $G$-algebras appears in that way.
\end{proof}

Let $k(X)$ be an indecomposable $G$-algebra. By choosing an element $p\in X$, we can identify the $G$-set $X$ with the set $G/H$ of cosets modulo the stabiliser subgroup $H=\mathrm{St}_G(p)$.

Let $\C(G,H,\alpha)$ be the category $\cC(F,G,\gamma,\beta,\alpha)$ with $(\gamma,\beta,\alpha)=\tau(\alpha)$ as defined in the appendix.

\begin{theo}\label{loc1}
Let $G$ be a finite group, and let $H\subset G$ be a subgroup.  The category $\Z(G,\alpha)_{k(G/H)}$ of right modules over the function algebra $k(G/H)$ in the Drinfeld center $\Z(G,\alpha)$ is equivalent, as a monoidal category, to the category $\C(G,H,\alpha)$.  
\nl
Moreover, the full subcategory $\Z(G,\alpha)_{k(G/H)}^\mathrm{loc}$ of local modules is equivalent, as a braided monoidal category, to the Drinfeld center $\Z(H,\alpha|_H)$.
\end{theo}
\begin{proof}
We will exhibit the claimed equivalence of categories by constructing a pair of quasi-inverse functors
$$\xymatrix{\Z(G,\alpha)_{k(G/H)} \ar@/^15pt/[r]^{D} & \C(G,H,\alpha) \ar@/^15pt/[l]^{E} }$$
To define the first functor $D$, let us choose a minimal idempotent $p$ in the function algebra $k(G/H)$ to be the indicator function on $H$:
$$p:G\rightarrow k,\qquad p(x)=\left\{\begin{array}{cc}1,&x\in H\\0,&\mathrm{otherwise}\\\end{array}\right.$$
Define $D:\Z(G,\alpha)_{k(G/H)}\rightarrow\C(G,H,\alpha)$ by $D(M)=Mp.$
Since $p$ is of degree zero, $Mp$ is a $G$-graded vector space in a natural way: $\left(Mp\right)_g = \left(M_g\right)p$. The $G$-action on $M$ reduces to an $H$-action on $Mp$. This makes $D(M)$ an object of $\C(G,H,\alpha)$. Clearly $D$ is a functor.  
If $M$ and $N$ are objects of $\Z(G,\alpha)_{k(G/H)}$, then $$D\left(M\otimes_{k(G/H)}N\right) = \left(M\otimes_{k(G/H)}N\right)p = \left(M\otimes_{k(G/H)}N\right)p^2 = $$ $$ = \left(Mp\right)\otimes_{k(G/H)}\left(Np\right) =  D(M)\otimes D(N)\ .$$ Thus $D$ is a tensor functor.

The second functor $E$ requires more preparation.
For $V\in \C(G,H,\alpha)$ let $\mathrm{Map}\left(G,V\right)$ be the vector space of set-theoretic maps $G\to V$. Define $G$-grading on $\mathrm{Map}\left(G,V\right)$ by 
\begin{equation}\label{mapdeg}
\left|a\right|=f\in G\iff\left|a(x)\right|=x^{-1}fx \quad\forall x\in G,
\end{equation}
Define an $\alpha$-projective $G$-action on $\mathrm{Map}\left(G,V\right)$ as follows. For a homogeneous  $a\in \mathrm{Map}\left(G,V\right)$ of degree $|a|=f$, define $g.a:G\rightarrow V$ by 
$$\left(g.a\right)(x)=\alpha\left(x^{-1},g|f\right)^{-1}a\left(g^{-1}x\right).$$
It is straightforward to check that $\mathrm{Map}\left(G,V\right)$ is an object of $\Z(G,\alpha)$.
\nl
Now consider a subspace of $\mathrm{Map}\left(G,V\right)$ given by 
\beq\lb{e}
E(V) = \{a:G\to V|\ a(xh) = \alpha(h^{-1},x^{-1}|f)h^{-1}.a(x),\ h\in H, x\in G, |a|=f\},
\eeq 
for $V\in\C(G,H,\alpha)$. 
It also is not hard to see that $E(V)$ is an object of $\Z(G,\alpha)$.  
\nl
Moreover $E(V)$ is a right $k(G/H)$-module, with the action $\nu:E(V)\otimes E(k)\rightarrow E(V)$ in $\Z(G,\alpha)$ defined by $\nu\left(a\otimes b\right)=ab$, where $(ab)(x)=a(x)\cdot b(x)$ (here we use that $E(k)=k(G/H)$).  
This makes $E$ a functor $\C\left(G,H,\alpha\right)\to\Z\left(G,\alpha\right)_{k(G/H)}$. 
\nl 
For $V,W\in \C(G,H,\alpha)$ the universal property of tensor product gives an isomorphism 
$$E(V)\otimes_{k(G/H)}E(W) = E(V)\otimes_{ E(k)}E(W) \simeq E(V\ot W),$$ 
which shows that $E$ is a monoidal functor.

It remains to define the adjunction isomorphisms $\phi:D\circ E\to Id$ and $\psi:Id\to E\circ D$.  
For $V\in\C(G,H,\alpha)$, define a map $\phi_V: E(V)p\rightarrow V$ by $\phi\left(ap\right)=a(e)$, where $e\in G$ is the identity element.  
For  $M\in\Z(G,\alpha)_{k(G/H)}$ define a map $\psi_M:M\rightarrow E(Mp)$ by $\psi\left(m\right)(x)=\left(x^{-1}.m\right)p$, $x\in G$.
It is a direct task to check that these are natural isomorphisms of functors.

Finally, we address the locality statement. 
For an object $V\in\Z(G,\alpha)$, denote by $\mathrm{supp}(V) = \{g\in G |\ V_g\not= 0\ \}$ the {\em support} of $V$. 
Let $M$ be a right $k(G/H)$-module, and $p$ be as above.  The support of $D(M)=Mp$ is a subset of $H$.
Indeed the locality condition implies (and is equivalent to the fact) that for $m\in M_g$ one has $mp=m(g.p)$. Hence $mp=mp^2=m((g.p)p)$, which for non-zero $m$ implies that $g.p=p$.
Note that the full subcategory of $\C(G,H,\alpha)$ of objects with support in $H$ is $\C(H,H,\alpha) = \Z(H,\alpha|_H)$.
Thus the restriction of $D$ to $\Z(G,\alpha)_{k(G/H)}^\mathrm{loc}$ lands in $\Z(H,\alpha|_H)$. It is straightforward to see that this restriction is braided.
\end{proof}

\bre Applied to (etale) algebras, the functor $E$ from the proof of theorem \ref{loc1} is {\em transfer},
defined in \cite{tu}. The transfer turns an etale algebra from $\cZ\left(H,\alpha|_H\right)$ into an algebra from $\Z(G,\alpha)$.
Indeed, by theorem \ref{loc1} for an etale algebra $B$ from $\cZ\left(H,\alpha|_H\right)$ the transfer $E(B)$ is an  etale algebra in ${\Z(G,\alpha)}^\mathrm{loc}_{k(G/H)}$, which by lemma \ref{algloc} is an etale algebra in $\Z(G,\alpha)$ containing $k(G/H)$.
\ere

\begin{corr}\label{trans}
For a simple separable algebra $A$ in $\Z(G,\alpha)$, there is a subgroup $H\subset G$ such that $A \simeq E(B)$, where $B$ is a simple separable algebra in $\cZ\left(H,\alpha|_H\right)$ with $B_e=k$.
\end{corr}
\begin{proof}
The subalgebra $A_e$ is an indecomposable commutative $G$-algebra. By lemma \ref{comgalg}, it is isomorphic to $k(X)$ for some transitive $G$-set $X$. By lemma \ref{algloc}, $A$ is a commutative algebra in ${\Z(G,\alpha)}^\mathrm{loc}_{A_e}$. Thus, by theorem \ref{loc1}, $A$ is the transfer of the etale algebra $B=Ap$ from $\cZ\left(H,\alpha|_H\right)$ (here $p$ is the minimal idempotent of $A_e$, corresponding to an element of $X$, with the stabilizer $H=\mathrm{St}_G(p)$). Finally, $B_e = A_{e}p = k$ by minimality of $p$.
\end{proof}

\subsection{Etale algebras trivial in trivial degree}\lb{eatt}

Here we describe simple etale algebras $B$ in $\cZ\left(H,\beta\right)$ with $B_e=k$.
\begin{lem}
Let $B$ be a separable algebra in $\cZ\left(H,\alpha|_H\right)$ such that $B_e=k$. Then $$dim\left(B_h\right)\leq 1,\quad \forall h\in H.$$
Moreover, the support $F$ of $B$ is a normal subgroup of $H$.
\end{lem}
\begin{proof}
By proposition \ref{sep}, an algebra $B$, such that $B_e=k$, is separable iff the multiplication defines the
non-degenerate pairing $m:B_g\otimes B_{g^{-1}}\to A_e = k$. Thus, $\alpha$-associativity of multiplication implies
that, for any $a,c\in B_g$ and $b\in B_{g^{-1}}$ $$\alpha(g,g^{-1},g)m(a,b)c = am(b,c).$$ For non-zero $a,c$, choosing
$b$ such that $m(a,b),m(b,c)\not=0$, we get that $a$ and $c$ are proportional.
\newline
Now, it follows from the non-degeneracy of $m:B_g\otimes B_{g^{-1}}\to A_e = k$, that a generator of a non zero $B_f$
is invertible. Thus, for non-zero components $B_f,B_g$ the product $B_fB_g$ is also non-zero.
\end{proof}

Let $F\vartriangleleft H$ be a normal subgroup and $\gamma\in C^2(F,k^*)$ be a cochain, such that $d(\gamma) = \alpha|_F$.
Denote by $k[F,\gamma]$ an $H$-graded $\alpha$-associative algebra with the basis $e_f, f\in F$, graded as $|e_f| = f$,
and with multiplication defined by $e_fe_g = \gamma(f,g)e_{fg}$.
\begin{prop}\label{ttd}
An indecomposable separable algebra $B$ in $\Z(H,\alpha|_H)$ with $B_e=k$ has the form $k[F,\gamma]$ with the
$\alpha$-projective $H$-action given by: $$h(e_f) = \varepsilon_h(f)e_{hfh^{-1}},$$ for some $\varepsilon:H\times F\to
k^*$ satisfying
\begin{equation}\label{eps1}
\varepsilon_{gh}(f)  =  \varepsilon_g(hfh^{-1})\varepsilon_h(f)\alpha(f|g,h),\quad g,h\in H, f\in F
\end{equation}
\begin{equation} \label{eps2}
\gamma(f,g)\varepsilon_h(fg)  =  \alpha(f,g|h)\varepsilon_h(f)\varepsilon_h(g)\gamma(hfh^{-1}hgh^{-1}),\quad h\in H, f,g\in F
\end{equation}
The algebra $k[F,\gamma]$ is commutative (and hence etale) if
\begin{equation}\label{eps3}
\gamma(f,g)  =  \varepsilon_f(g)\gamma(fgf^{-1},f),\quad f,g\in F.
\end{equation}
\end{prop}
\begin{proof}
 Indeed, $\alpha$-projectivity of the action requires that
$(gh)(e_f) = \varepsilon_{gh}(f)e_{ghfh^{-1}g^{-1}}$ coincides with $$\alpha(f|g,h)g(h(e_f)) =
\alpha(f|g,h)\varepsilon_h(f)\varepsilon_g(hfh^{-1}e_{ghfh^{-1}g^{-1}},$$ which gives the first identity.
Multiplicativity of the action amounts to the equality between $$h(e_fe_g) =
\gamma(f,g)\varepsilon_h(fg)e_{hfgh^{-1}}$$ and $$\alpha(f,g|h)h(e_f)h(e_g) =
\alpha(f,g|h)\varepsilon_h(f)\varepsilon_h(g)\gamma(hfh^{-1},hgh^{-1})e_{hfgh^{-1}},$$ which gives the second identity.
Finally, commutativity implies that $e_fe_g = \gamma(f,g)e_{fg}$ is equal to
$$
f(e_g)e_f = \varepsilon_f(g)e_{fgf^{-1}}e_f = \varepsilon_f(g)\gamma(fgf^{-1},f)e_{fg}.
$$
\end{proof}
Denote by $k[F,\gamma,\varepsilon]$ the etale algebra in $\Z(H,\alpha|_H)$, defined in
proposition \ref{ttd}.
\begin{lem}
Two algebras $k[F,\gamma,\varepsilon]$ and $k[F',\gamma',\varepsilon']$ in the category $\Z(H,\alpha|_H)$ are
isomorphic iff there is a cochain $c:F\to k^*$ such that
\beq\lb{coalg}
c(fg)\gamma(f,g) = \gamma'(f,g)c(f)c(g),\quad \varepsilon_h(f)c(hfh^{-1}) = c(f){\varepsilon '}_h(g).
\eeq
\end{lem}
\begin{proof}
Isomorphic algebras in $\Z(H,\alpha|_H)$ have to have the same supports. Thus $F=F'$. Since the components of both
$k[F,\gamma,\varepsilon]$ and $k[F',\gamma',\varepsilon']$ are all one dimensional, an isomorphism
$k[F,\gamma,\varepsilon]\to k[F',\gamma',\varepsilon']$ has a form $e_f\mapsto c(f)e_f$ for some $c(f)\in k^*$.
Finally, multiplicativity of this mapping is equivalent to the first condition, while $H$-equivariance is equivalent to
the second.
\end{proof}

Note that $\varepsilon$ can be considered as an element of $C^1(H,C^1(F,k^*))$, while $\gamma$ lies naturally in
$C^2(F,k^*)=C^0(H,C^2(F,k^*))$. Thus, in the terminology of the appendix, $(\varepsilon,\gamma)$ is a 2-cochain of
$\tilde C^*(H,F,k^*)$. The equations (\ref{eps1}),(\ref{eps2}), together with the condition $d(\gamma)=\alpha|_F$, are
equivalent to the coboundary condition $d(\varepsilon,\gamma)=(\alpha_2,\alpha_1,\alpha_0)=\tau(\alpha)$ in $\tilde
C^*(H,F,k^*)$. The equations (\ref{coalg}) say that $(\varepsilon,\gamma)=d(c)(\varepsilon',\gamma')$ for $c\in
C^1(F,k^*)=\tilde C^1(H,F,k^*)$.

Before we describe local modules over the algebras, defined in proposition \ref{ttd}, we need to explain how the
cochains $\gamma,\varepsilon$, associated with them, allow to reduce the cocycle $\alpha\in Z^3(H,k^*)$ to
$\overline\alpha\in Z^3(H/F,k^*)$. It will be shown, in the course of the proof of theorem \ref{loc2}, that
$\overline\alpha(x,y,z)$ defined by
$$\alpha(s(x),s(y),s(y)^{-1}s(x)^{-1}s(xyz))\gamma(\tau(y,z),\tau(x,yz)))\gamma(\tau(y,z)\tau(x,yz),\tau'(x,y,z))\times$$
\begin{equation}\label{3coc}
\times\varepsilon_{s(xyz)^{-1}s(x)s(y)}(\tau(x,y))\gamma(\tau'(x,y,z)^{-1},\tau'(x,y,z))
\end{equation}
is a 3-cocycle of $H/F$. Here $s:H/F\to H$ is a section of the quotient map $H\to H/F$, $\tau(y,z) = s(z)^{-1}s(y)^{-1}s(yz)\in F$ and
$$\tau'(x,y,z) = s(xyz)^{-1}s(x)s(y)\tau(x,y)^{-1}s(y)^{-1}s(x)^{-1}s(xyz).$$

\begin{theo}\label{loc2}
The category ${\Z(H,\alpha|_H)}^\mathrm{loc}_{k[F,\gamma,\varepsilon]}$, of local right $k[F,\gamma,\varepsilon]$-modules in
$\Z(H,\alpha|_H)$, is equivalent, as a ribbon category, to $\\Z(H/F,{\overline\alpha})$.
\end{theo}
\begin{proof}
The structure of a right $k[F,\gamma,\varepsilon]$-module on an object $M=\oplus_{h\in H}M_h$ of $\Z(H,\alpha|_H)$
amounts to a collection of isomorphisms $e_f:M_h\to M_{hf}$ (right multiplication by $e_f\in k[F,\gamma,\varepsilon]$)
such that $$e_e = I,\quad e_fe_{f'} = \gamma(f,f')e_{f'f},\quad he_fh^{-1} = \varepsilon_h(f)e_{hfh^{-1}},\quad f,f'\in
F, h\in H.$$ Here $h:M_{h'}\to M_{hh'h^{-1}}$ is the $\alpha$-projective $H$-action on $M$. The
$k[F,\gamma,\varepsilon]$-module $M$ is local iff $e_f = \varepsilon_h(f)hfh^{-1}e_{hfh^{-1}}$ on $M_h$. Indeed, the
double braiding in $\Z(H,\alpha|_H)$ transforms an element $m\otimes e_f\in M\otimes A$ (with $m\in M_h$) as follows
$$m\otimes e_f\mapsto h(e_f)\otimes m=\varepsilon_h(f)e_{hfh^{-1}}\otimes m\mapsto\varepsilon_h(f)hfh^{-1}(m)\otimes
e_{hfh^{-1}}.$$ An equivalent way of expressing the locality condition is the following:
$$f=\varepsilon_h(h^{-1}fh)^{-1}\gamma(h^{-1}fh,f^{-1})\gamma(f,f^{-1})^{-1}e_{h^{-1}fhf^{-1}}=\epsilon_h(f)e_{[h^{-1},f]}.$$
\newline
Now let $s:H/F\to H$ be a section of the quotient map $H\to H/F$. For a $k[F,\gamma,\varepsilon]$-module $M$ define a
$H/F$-graded vector space $V$ by $V_x = M_{s(x)}$, where $x\in H/F$. For local $M$ the vector space $V$ can be equipped
with a projective $H/F$-action: for $y\in H/F$ define $y:V_x\to V_{yxy^{-1}}$ as the composition
$$\xymatrix{ V_x=M_{s(x)} \ar[r]^(.4){s(y)} & M_{s(y)s(x)s(y)^{-1}} \ar[r]^(.4){e_{f(x,y)}} & M_{s(yxy^{-1})}=V_{yxy^{-1}}, }$$
where $$f(x,y)=s(y)s(x)^{-1}s(y)^{-1}s(yxy^{-1})={^{s(y)}{s(x)}^{-1}}s({^{y}{x}})\in F.$$ To compute the projective
multiplier one would need to compare $zy$ on $V_x$ with the composition of $y$ and $z$. This can be done with the help
of the following diagram:
$$\xymatrix{V_x \ar[rr]^y \ar@{=}[ddd] && V_{^{y}{x}} \ar[rr]^z \ar@{=}[d] && V_{^{zy}{x}} \ar@{=}[ddd] \\
&& M_{s({^{y}{x}})} \ar[rd]^{s(z)} &&  \\
 & M_{^{s(y)}{s(x)}} \ar[ru]^{e_{f(x,y)}} \ar[rd]^{s(z)} && M_{^{s(z)}{s({^{y}{x}})}} \ar[rd]^{e_{f({^{y}{x}},y)}} \\
M_{s(x)} \ar[ur]^{s(y)} \ar[rr]^{s(z)s(x)} \ar@/_10pt/[drr]_{s(zy)}  && M_{^{s(z)s(y)}{s(x)}} \ar[ru]^{e_{^{s(z)}{f(x,y)}}}
\ar[rr]^{e_{g(x,y,z)}} && M_{s({^{zy}{x}})}\\
 && M_{^{s(zy)}{s(x)}} \ar@/^10pt/[u]^{\sigma(z,y)} \ar@/_10pt/[u]_{e_{[{^{s(zy)}{s(x)}}^{-1},\sigma(z,y)]}} \ar@/_10pt/[urr]_{f(x,zy)}\\
V_x \ar@{=}[uu] \ar[rrrr]^{zy} &&&& V_{^{zy}{x}} \ar@{=}[uu]
}$$ Here $\sigma(z,y) = s(z)s(y)s(zy)^{-1}\in F$ and $$g(x,y,z)={^{s(z)}{f(x,y)}}f({^{y}x},z)={^{s(z)s(y)}{s(x)}^{-1}}s({^{zy}{x}})$$
so that
$$[{^{s(zy)}{s(x)}}^{-1},\sigma(z,y)]{^{s(z)s(y)}{s(x)}^{-1}}s({^{zy}{x}})$$ coincides with
$${^{s(zy)}{s(x)}^{-1}}s({^{zy}{x}})=f(x,zy).$$
The cells of the diagram commute up to multiplication by a scalar (except two top and one bottom cells, which commute
on the nose). Carefully gathering the scalars one can write down the multiplier for the projective $H/F$-action on $V$.
Fortunately, we do not have to do it. In view of the proposition \ref{chace} from the appendix, it is enough to know
that such a multiplier exists.
\newline
The $H/F$-graded vector space $V\otimes U$, corresponding to the tensor product $M\otimes_{k[F,\gamma,\varepsilon]}N$
of local $B=k[F,\gamma,\varepsilon]$-modules, can be identified with the graded tensor product of $V$ and $U$. Indeed,
the composition
$$\xymatrix{\bigoplus\limits_{yz=x}M_{s(y)}\otimes N_{s(z)} \ar[rr]^(.4){1\otimes e_{\tau(y,z)}} && \bigoplus\limits_{fg=s(x)}M_f\otimes N_g = (M\otimes N)_{s(x)} \ar[r]^(.65){pr} &  (M\otimes_BN)_{s(x)}  }$$ defines an isomorphism
$\bigoplus_{yz=x}V_y\otimes U_z\to (V\otimes U)_x$. Here, as before, $\tau(y,z) = s(z)^{-1}s(y)^{-1}s(yz)\in F$. Again
we will use a diagrammatic language to prove the compatibility (up to a scalar) of the $H/F$-action on $V\otimes U$
with the $H/F$-actions on $V$ and $U$:
$$\xymatrix{(V\otimes U)_x \ar[rr]^y \ar@{=}[d] && (V\otimes U)_{yxy^{-1}} \ar@{=}[d] \\
(M\otimes_BN)_{s(x)} \ar[r]^{s(y)} & (M\otimes_BN)_{s(y)s(x)s(y)^{-1}} \ar[r]_{e_{f(x,y)}} &
(M\otimes_BN)_{s(yxy^{-1})}\\ (M\otimes N)_{s(x)} \ar[u]_{pr} \ar[r]^{s(y)} \ar@{=}[d] & (M\otimes
N)_{s(y)s(x)s(y)^{-1}} \ar[u]_{pr} \ar@{=}[d] & (M\otimes N)_{s(yxy^{-1})} \ar[u]_{pr} \ar@{=}[d]\\
\bigoplus\limits_{fg=s(x)}M_f\otimes N_g \ar[r]^(.4){s(y)\otimes s(y)} &
\bigoplus\limits_{fg=s(x)}M_{s(y)fs(y)^{-1}}\otimes N_{s(y)gs(y)^{-1}} \ar[r]^(.64){e_{f(v,y)}\otimes e_{h(v,y,g)}} &
\bigoplus\limits_{f'g'=s(yxy^{-1})}M_{f'}\otimes N_{g'}\\ \bigoplus\limits_{vu=x}M_{s(v)}\otimes N_{s(u)}
\ar[u]_{1\otimes e_{\tau(v,u)}} \ar[r]^(.38){s(y)\otimes s(y)} \ar@{=}[d] &
\bigoplus\limits_{vu=x}M_{s(y)s(v)s(y)^{-1}}\otimes N_{s(y)s(u)s(y)^{-1}} \ar[u]_{1\otimes e_{s(y)\tau(v,u)s(y)^{-1}}}
& \bigoplus\limits_{vu=x}M_{s(yvy^{-1})}\otimes N_{s(yuy^{-1})} \ar@{=}[d] \ar[u]^{1\otimes
e_{\tau(yvy^{-1},yuy^{-1})}}\\ \bigoplus\limits_{vu=x}V_v\otimes U_u \ar[rd]_{y\otimes y}  &
\bigoplus\limits_{vu=x}M_{s(yvy^{-1})}\otimes N_{s(yuy^{-1})} \ar@{=}[d] \ar[u]_{e_{f(v,y)}^{-1}\otimes
e_{f(u,y)}^{-1}} & \bigoplus\limits_{vu=x}V_{yvy^{-1}}\otimes U_{yuy^{-1}} \\ &
\bigoplus\limits_{vu=x}V_{yvy^{-1}}\otimes U_{yuy^{-1}} \ar@{=}[ru] }$$ Here
$h(v,y,g)=s(y)gs(y)^{-1}f(v,y)^{-1}s(y)g^{-1}s(y)^{-1}$. Again, the cells of the diagram commute up to scalars (one for
each $v$ and $u$), resulting in a overall factor, rescaling $y\otimes y$ into $y$ on $V\otimes U$. Note that the six
vertex cell in the right half of the diagram commutes by the following property of the projection map: for any $u\in F$
the diagram
$$\xymatrix{ & (M\otimes_BN) \\ (M\otimes N) \ar[ru]^{pr} \ar@{=}[d] && (M\otimes N) \ar[lu]_{pr} \ar@{=}[d]\\
\bigoplus\limits_{f,g}M_f\otimes N_g \ar[rr]_{\oplus e_u\otimes e_{g^{-1}u^{-1}g}} &&
\bigoplus\limits_{f',g'}M_{f'}\otimes N_{g'} }$$ commutes up to multiplication by scalars (one for each $f$ and $g$).
So far, their actual form has been unimportant, but it will become crucial in what we are going to do later. To calculate
this factor, we start with the identity
$$me_u\otimes n = \varepsilon_{h^{-1}}(u)(m\otimes ne_{h^{-1}uh}),$$ which follows from the definition of the tensor product over $B$.
Multiplying this identity with $e_{h^{-1}u^{-1}h}$ (from the right) we get
$$me_u\otimes ne_{h^{-1}u^{-1}h} = \varepsilon_{h^{-1}}(u)\gamma(h^{-1}uh,h^{-1}u^{-1}h)(m\otimes n).$$
\newline
Tha last step we need to make is to calculate the associator on $V\otimes(U\otimes W)$ and to prove that it is scalar
on $V_x\otimes(U_y\otimes W_z)$. Once again we apply diagrammatic technique:
$$\xymatrix@C=-45pt{V_x\otimes(U_y\otimes W_z) \ar@{=}[d] \ar[rr]^{\overline\alpha(x,y,z)} && (V_x\otimes U_y)\otimes W_z
\ar@{=}[d]\\ M_{s(x)}\otimes(N_{s(y)}\otimes L_{s(z)}) \ar[d]^{1\otimes 1\otimes e_{\tau(y,z)}} && (M_{s(x)}\otimes
N_{s(y)})\otimes L_{s(z)} \ar[d]^{1\otimes e_{\tau(x,y)}\otimes 1}\\ M_{s(x)}\otimes(N_{s(y)}\otimes
L_{s(y)^{-1}s(yz)}) \ar[d]^{1\otimes 1\otimes e_{\tau(x,yz)}} && (M_{s(x)}\otimes N_{s(x)^{-1}s(xy)})\otimes L_{s(z)}
\ar[d]^{1\otimes 1\otimes e_{\tau(xy,z)}}\\ M_{s(x)}\otimes(N_{s(y)}\otimes L_{s(y)^{-1}s(x)^{-1}s(xyz)})
\ar@{_{(}->}[dd] \ar[rd]^{a} && (M_{s(x)}\otimes N_{s(x)^{-1}s(xy)})\otimes L_{s(xy)^{-1}s(xyz)} \ar@{_{(}->}[dd]\\ &
(M_{s(x)}\otimes N_{s(y)})\otimes L_{s(y)^{-1}s(x)^{-1}s(xyz)} \ar[ur]_{*} \ar@{_{(}->}[d] & \\ M\otimes(N\otimes L)
\ar[d]^{pr} \ar[r]_{\alpha_{M,N,L}} & (M\otimes N)\otimes L \ar[d]^{pr} & (M\otimes N)\otimes L \ar[dl]^{pr} \\
M\otimes_B(N\otimes_B L) \ar[r]_{\alpha_{M,N,L}} & (M\otimes_B N)\otimes_B L & }$$ Here $a$ stands for the
multiplication by $\alpha(s(x),s(y),s(y)^{-1}s(x)^{-1}s(xyz))$, $$*=
\varepsilon_{s(xyz)^{-1}s(x)s(y)}(\tau(x,y))\gamma(\tau'(x,y,z)^{-1},\tau'(x,y,z))(1\otimes e_{\tau(x,y)}\otimes
e_{\tau'(x,y,z)}),$$ and again
$$\tau'(x,y,z) = s(xyz)^{-1}s(x)s(y)\tau(x,y)^{-1}s(y)^{-1}s(x)^{-1}s(xyz).$$
Now, composing the arrows of the top cell and comparing the coefficients gives the formula (\ref{3coc}). Finally, by proposition \ref{chace} from the Appendix, the category ${\Z(H,\alpha|_H)}^\mathrm{loc}_{k[F,\gamma,\varepsilon]}$ is ribbon equivalent to $\Z(H/F,{\overline\alpha})$.
\end{proof}

\begin{corr}\label{lagr1}
An etale algebra $B=k\left[F,\gamma,\epsilon\right]$ in the category $\Z(H,\alpha|_H)$ is Lagrangian if and only if $F=H$.
\end{corr}
\begin{proof}
By theorem \ref{loc2}, an etale algebra $B=k\left[F,\gamma,\epsilon\right]$ in $\Z(H,\alpha|_H)$ is Lagrangian if and only if $\mathcal{Z}\left(H/F,\overline{\alpha}\right)\simeq k$-$\Vect$,  {\em i.e.} if and only if the quotient group $H/F$ is trivial.
\end{proof}

Note that for $F=H$, $\epsilon$ is determined by $\gamma$  by  the equation \eqref{eps3}. 

\subsection{Etale algebras and their local modules}\lb{ealm}

In this section we combine the previous results on etale algebras in group-theoretical modular
categories and on their local modules.

Define 
\beq
A(H,F,\gamma,\varepsilon) = E(k[F,\gamma,\varepsilon])\ ,
\eeq
where $E:\Z(H,\alpha|_H)\to \Z(G,\alpha)_{k(G/H)}$ is the functor from the proof of theorem \ref{loc1}.

\begin{theo}\label{main}
An etale algebra in $\Z(G,\alpha)$ has the form $A(H,F,\gamma,\varepsilon)$, where $H\subset G$ is a subgroup, $F\unlhd H$ is a normal subgroup, $\gamma\in C^2(F,k^*)$ is a coboundary $d(\gamma) = \alpha|_F$ and $\varepsilon:H\times F\to k^*$ satisfies the conditions (\ref{eps1},\ref{eps2},\ref{eps3}).
\end{theo}
\begin{proof}
Follows from corollary \ref{trans} and lemma \ref{algloc}.
\end{proof}

\begin{theo}\label{local}
The category ${\Z(G,\alpha)}^{\mathrm{loc}}_{A(H,F,\gamma,\varepsilon)}$, of local right $A(H,F,\gamma,\varepsilon)$-modules in $\Z(G,\alpha)$, is equivalent, as a ribbon category, to $\cZ\left(H/F,\overline{\alpha}\right)$.
\end{theo}
\begin{proof}
Follows from theorems \ref{loc1} and \ref{loc2}, and lemma \ref{algloc}.
\end{proof}

Note that when $F=H$ the function $\varepsilon$ is completely determined by $\gamma$. Indeed, by the commutativity condition (\ref{eps3}) , one has
\beq\lb{atga}\epsilon_f(g) = \frac{\gamma\left(f,g\right)}{\gamma\left(fgf^{-1},f\right)}\ .
\eeq
\nl
Denote $A(H,H,\gamma,\varepsilon)$ by $L\left(H,\gamma\right)$. 
Theorems \ref{main} and \ref{local} allow us to describe Lagrangian algebras in group-theoretical modular categories in purely group-theoretical terms.  

\bco\label{Lagr}
A Lagrangian algebra $L\in\Z(G,\alpha)$ has the form $L(H,\gamma)$, where $H\subset G$ is a subgroup and $\gamma\in C^2(H,k^*)$ is a coboundary $d(\gamma) = \alpha|_H$.  
\eco
\begin{proof}
Follows from corollary \ref{lagr1} and theorems \ref{main} and \ref{local}.
\end{proof}

\bre
A Lagrangian algebra $L=L\left(H,\gamma\right)$ is completely characterised by the conditions
$$L_e\simeq k(G/H), \qquad D(L) = k\left[H,\gamma\right]\ \in\ \C(G,H,\alpha).$$
Here $D:\Z(G,\alpha)_{k(G/H)}\to\cC\left(G,H,\alpha\right)$ is the functor from the proof of theorem \ref{loc1}.
\ere

\bre
It follows from the corollary \ref{Lagr} that Lagrangian algebras in $\Z(G,\alpha)\boxtimes
\Z(G,\alpha^{-1})\simeq \cZ(G\times G,\alpha\times\alpha^{-1})$ correspond to pairs $(U,\gamma)$, where $U\subset
G\times G$ is a subgroup and $\gamma\in C^2(U,k^*)$ is a coboundary $d(\gamma) = (\alpha\times\alpha^{-1})|_U$. This
coincides with the parametrisation of module categories obtained in \cite{os}, which illustrates the fact  that the total centre defines a bijection between equivalence classes of indecomposable
module categories over $\Z(G,\alpha)$ and Lagrangian algebras in
$\Z(G,\alpha)\boxtimes\Z(G,\alpha^{-1})$.
\ere

\section{Full centre}\lb{fc}

Here we show that Lagrangian algebras in $\Z(G,\alpha)$ are full centres of separable indecomposable algebras in $\V(G,\alpha)$.

\subsection{Monoidal centre of $\V(G,\alpha)$}\lb{mcgs}

Denote by $\V(G,\alpha)$ the category of $G$-graded vector spaces with
the ordinary tensor product: $$(V\otimes U)_g = \bigoplus_{fh=g}V_f\otimes U_h$$ and the associativity constraint
$\phi_{V,U,W}:V\otimes(U\otimes W)\to (V\otimes U)\otimes W$, twisted by a 3-cocycle $\alpha\in Z^3(G,k^*)$:
$$\phi_{V,U,W}(v\otimes(u\otimes w)) = \alpha(f,g,h)(v\otimes u)\otimes w,\quad \forall v\in V_f,u\in U_g,w\in W_h.$$
Clearly, $\V(G,\alpha)$ is a fusion category with the set of simple objects $\mathrm{Irr}\left(\V(G,\alpha)\right)=G$. We denote by $I(g)$ the simple object corresponding to $g\in G$: the one dimensional graded vector space concentrated in degree $g$.

Here we describe the monoidal centre of $\V(G,\alpha)$. 
\begin{prop}\label{mc}
The monoidal centre $\cZ(\V(G,\alpha))$ is isomorphic, as braided monoidal category, to the category $\Z(G,\alpha)$.
\nl
The canonical forgetful functor $\cZ(\V(G,\alpha))\to \V(G,\alpha)$ corresponds to the functor $\Z(G,\alpha)\to \V(G,\alpha)$ forgetting the $G$-action.
\end{prop}
\begin{proof}
For an object $(X,x)$ of the centre $\cZ(\V(G,\alpha))$, the natural isomorphism $$x_V:V\otimes X\to
X\otimes V,\quad V\in \V(G,\alpha)$$ is defined by its evaluations on one-dimensional graded vector spaces. 
The isomorphism $x_{I(f)}$ can be
seen as an automorphism $f:X\to X$. The fact, that $x_{I(f)}$ preserves grading, amounts to the condition $f(X_g)
= X_{fgf^{-1}}$: $$\xymatrix{X_g = (I(f)\otimes X)_{fg} \ar[rr]^{x_{I(f)}} && (X\otimes I(f))_{fg} = X_{fgf^{-1}}\
.}$$ The coherence condition for $x$ is equivalent to the the equation (\ref{pa}), with the associativity
constraints giving rise to $\alpha(h|f,g)$. The diagram, defining the second component $x|y$ of the tensor
product $(X,x)\otimes(Y,{\it y}) = (X\otimes Y,x|y)$, is equivalent to the tensor product of
projective actions (\ref{tp}), with the associativity constraints giving rise to $\alpha(g,h|f)$.
\newline
The description of the monoidal unit in a monoidal centre corresponds to the answer for the monoidal unit in
$\Z(G,\alpha)$.
\newline
Clearly, the braiding $c_{(X,x),(Y,y)} = y_X$ of the monoidal centre $\cZ(\V(G,\alpha))$ corresponds to the braiding
(\ref{br}) of $\Z(G,\alpha)$.
\end{proof}

\subsection{Full centre}

Here we identify Lagrangian algebras in $\Z(G,\alpha)$ with full centres of separable indecomposable algebras in $\V(G,\alpha)$.
Recall that any such algebra is isomorphic to the twisted group algebra  $k[H,\gamma]$ for a subgroup $H\subset G$ and a coboundary $d(\gamma)=\alpha|_H$. 

\bth\lb{fcia}
The full centre $Z(k[H,\gamma])$ coincides with $L(H,\gamma)$.
\eth
\bpf
It suffices to construct a homomorphism of algebras $\zeta:L(H,\gamma) \to k[H,\gamma]$ in $\V(G,\alpha)$ fitting in the diagram \eqref{cp}. Indeed, such a homomorphism induces a homomorphism of algebras $\tilde\zeta:L(H,\gamma) \to Z(k[H,\gamma])$ in $\Z(G,\alpha)$. Since $L(H,\gamma)$ is separable and indecomposable in $\Z(G,\alpha)$, by lemma \ref{is}, $\tilde\zeta$ is a monomorphism. Finally, the dimension of $\Z(G,\alpha)$ is $|G|$, which coincides with the dimension of the full centre $Z(k[H,\gamma])$. 
\nl
Let us define a map $\zeta:E\left(k\left[H,\gamma\right]\right)\rightarrow k\left[H,\gamma\right]$ by
$
\zeta\left(a\right)=a(e).
$
Observe that this definition, effectively, implies that $\zeta(a)=0$ if $|a|\in G\setminus H$.
We claim that $\zeta$ is a homomorphism of algebras in the category $\V\left(G,\alpha\right)$.  As $\zeta$ is an evaluation map, it is automatically multiplicative.  It remains to check that $\zeta$ is $G$-graded. Recall from the proof of theorem \ref{loc1} that for a homogeneous $a\in E\left(k\left[H,\gamma\right]\right)$ 
the degree of the values are $\left|a(x)\right|=x^{-1}|a|x\qquad \forall x\in G \ .$ 
Thus $\left|\zeta(a)\right|=\left|a(e)\right|=|a|\ .$
\nl
Now we consider the diagram (\ref{cp}).  Here, $A=k\left[H,\gamma\right]$ and $Z=L\left(H,\gamma\right)$, so that the diagram becomes
$$
\xymatrix{ k\left[H,\gamma\right]\otimes L\left(H,\gamma\right) \ar[r]^<<<<<{\mathrm{Id}\otimes\zeta} 
\ar[dd]_{c} & k\left[H,\gamma\right]^{\otimes2}  \ar[rd]^\mu\\ & & k\left[H,\gamma\right]\\ 
 L\left(H,\gamma\right)\otimes k\left[H,\gamma\right] \ar[r]_<<<<{\zeta\otimes\mathrm{Id}} & k\left[H,\gamma\right]^{\otimes2} \ar[ur]_\mu }
$$
Let $a\in L\left(H,\gamma\right)$, and let $e_h\in k\left[H,\gamma\right]$ for some $h\in H$.  Going the short way, we obtain
$$
e_h\otimes a\mapsto e_h\otimes a(e)\mapsto e_{h}a(e).
$$
Going the long way, we obtain 
$$
e_h\otimes a\mapsto\left(h.a\right)\otimes e_h\mapsto\left(\left(h.a\right)(e)\right)\otimes e_h\mapsto\left(\left(h.a\right)(e)\right)e_h.
$$
By $H$-equivariance of $a$, the last expression is 
$$
\left(\left(h.a\right)(e)\right)e_h=a\left(h^{-1}\right)e_h=\Big(h.\left(a(e)\right)\Big)e_h
$$
Finally, since the $H$-action on $k\left[H,\gamma\right]$ is inner, one has
$$\Big(h.\left(a(e)\right)\Big)e_h =  e_{h}a(e)e_h^{-1}e_h\ .$$
\epf

\section{Characters of Lagrangian algebras}

\subsection{Characters}\lb{char}

Here we recall basic facts about characters of objects of $\Z(G,\alpha)$. 
\nl
For an object $Z\in\Z(G,\alpha)$, define its {\em character} as the function on pairs of commuting elements of $G$ defined by
$$\chi_Z(f,g) = \mathrm{Tr}_{Z_f}(g).$$ 

\begin{lem}
The character $\chi_{Z}$ of $Z\in\Z(G,\alpha)$ satisfies
\beq\lb{pcf}\chi_{Z}\left(xfx^{-1},xgx^{-1}\right)=\frac{\alpha\left(x,g|f\right)}{\alpha\left(xgx^{-1},x|f\right)}\chi_{Z}\left(f,g\right)\ .
\eeq
\end{lem}
\begin{proof}
Let $\rho:G\rightarrow\mathrm{Aut}\left(Z\right)$ be the homomorphism corresponding to the action of $G$ on the vector space $Z$.  
Note that by \eqref{pa} we can write
$$
\rho(x)\rho(g)\rho(x)^{-1} = \frac{\alpha\left(xgx^{-1},x|f\right)}{\alpha\left(x,g|f\right)}\rho\left(xgx^{-1}\right),
$$
Indeed,
$\rho(x)\rho(g) = \alpha\left(x,g|f\right)^{-1}\rho(xg)$
together with 
$\rho\left(xgx^{-1}\right)\rho\left(x\right) = \alpha\left(xgx^{-1},x|f\right)^{-1}\rho(xg)$ give the desired.

Finally
$$\chi_{Z}\left(f,g\right)=\mathrm{Tr}_{Z_{f}}\left(\rho(g)\right)=\mathrm{Tr}_{Z_{xfx^{-1}}}\left(\rho(x)\rho(g)\rho(x)^{-1}\right) = $$ $$= \frac{\alpha\left(xgx^{-1},x|f\right)}{\alpha\left(x,g|f\right)}\mathrm{Tr}_{Z_{xfx^{-1}}}\left(\rho\left(xgx^{-1}\right)\right) = \frac{\alpha\left(xgx^{-1},x|f\right)}{\alpha\left(x,g|f\right)}\chi_{Z}\left(xfx^{-1},xgx^{-1}\right)\ .
$$
\end{proof}

By a {\em character} of $\Z(G,\alpha)$, we will mean a function on pairs of commuting elements of $G$ satisfying the projective class function property \eqref{pcf}. 

\bre
Equation \eqref{pcf} implies that a character of $\Z(G,\alpha)$ can be non-zero only on those commuting pairs $(f,g)$ for which 
$$\frac{\alpha(x,f,g)\alpha(g,x,f)\alpha(f,g,x)}{\alpha(x,g,f)\alpha(f,x,g)\alpha(g,f,x)} = 1$$
for any $x$ from the centraliser $C_G(f,g)$. 
\ere

For characters $\chi$ and $\xi$ of $\Z(G,\alpha)$, define their {\em product} by 
$$
\left(\chi\xi\right)\left(f,g\right)=\sum_{\substack{f_{1}f_{2}=f\\f_{i}g=gf_{i}}}{\alpha\left(g|f_{1},f_{2}\right)\chi\left(f_{1},g\right)\xi\left(f_{2},g\right)}.
$$
It can be checked that the product of characters is a character.

\begin{lem}
Let $Z,W\in\Z(G,\alpha)$, then 
$$
\chi_Z\chi_W=\chi_{Z\otimes W}.
$$
\end{lem}
\bpf
Write
$$\chi_{Z\otimes W}\left(f,g\right)=\mathrm{Tr}_{\left(Z\otimes W\right)_f}(g)=\sum_{\substack{f_{1}f_{2}=f\\f_{i}g=gf_i}}{\mathrm{Tr}_{Z_{f_1}\otimes W_{f_2}}(g)}\ .$$  Using \eqref{tp}, we get $$\sum_{\substack{f_{1}f_{2}=f\\f_{i}g=gf_i}}{\mathrm{Tr}_{Z_{f_1}\otimes W_{f_2}}(g)}=\sum_{\substack{f_{1}f_{2}=f\\f_{i}g=gf_i}}{\alpha\left(g|f_1,f_2\right)\mathrm{Tr}_{Z_{f_1}}(g)\mathrm{Tr}_{W_{f_2}}(g)}=$$ $$=\sum_{\substack{f_{1}f_{2}=f\\f_{i}g=gf_i}}{\alpha\left(g|f_1,f_2\right)\chi_{Z}\left(f_1,g\right)\chi_{W}\left(f_2,g\right)}=(\chi_{Z}\chi_{W})(f,g)\ .$$ 
\epf

\bre
Using lemma \ref{do}, one finds that the character of the dual object (the {\em dual character}) has the form
$$\chi_{Z^\sve}(f,g) = \frac{\alpha(g^{-1},g|f^{-1})}{\alpha(g|f,f^{-1})}\chi_Z(f^{-1},g^{-1}).$$
\ere

Define the {\em scalar product} of characters of $\Z(G,\alpha)$
(see \cite{ba1}):
$$(\chi,\psi) = \frac{1}{|G|}\sum_{\substack{f,g\in G, \\ fg=gf}}\alpha(g^{-1},g|f)\chi(f,g^{-1})\psi(f,g),$$ 
The scalar product calculates dimensions of corresponding Hom-spaces in $\Z(G,\alpha)$. 
\begin{lem}
Let $Z,W\in\Z(G,\alpha)$.  Then 
$$(\chi_Z,\chi_W) = \dim\left(\Z(G,\alpha)(Z,W)\right).$$ 
\end{lem}
\bpf
Note first that the Hom-space $\Z(G,\alpha)(I,W)$ coincides with the vector space of $G$-invariants $W_{e}^{G}$, so that
$$\dim\left(\Z(G,\alpha)(I,W)\right)  = dim_k(W_e^G) = \frac{1}{|G|}\sum_g \chi_W(e,g)\ .$$
In the general case, $\Z(G,\alpha)(Z,W) \simeq \Z(G,\alpha)(I,Z^\sve\ot W)$, and 
$$\dim\left(\Z(G,\alpha)(Z,W)\right) = \dim\left(\Z(G,\alpha)(I,Z^\sve\ot W)\right) = \frac{1}{|G|}\sum_g \chi_{Z^\sve\ot W}(e,g) = $$
$$= \frac{1}{|G|}\sum_{\substack{f,g\in G, \\ fg=gf}}\alpha(g|f^{-1},f)\chi_{Z^\sve}(f^{-1},g)\chi_W(f,g) = \frac{1}{|G|}\sum_{\substack{f,g\in G, \\ fg=gf}}\alpha(g|f^{-1},f)\frac{\alpha(g^{-1},g|f^{-1})}{\alpha(g|f,f^{-1})}\chi_Z(f,g^{-1})\chi_W(f,g) = $$
$$=  \frac{1}{|G|}\sum_{\substack{f,g\in G, \\ fg=gf}}\alpha(g^{-1},g|f)\chi_Z(f,g^{-1})\chi_W(f,g)\ .$$
\epf

In particular, for simple $Z,W\in\Z(G,\alpha)$ the scalar product $(\chi_Z,\chi_W) = 1$ iff $Z\simeq W$ and zero otherwise.

\subsection{Characters of Lagrangian algebras}

Here we compute the characters of Lagrangian algebras in $\Z(G,\alpha)$. 
\nl
Recall from the proof of theorem \ref{loc1} the functor $E:\Z(H,\alpha|_H)\to \Z(G,\alpha)$ given by 
\eqref{e}.

\begin{lem}\label{charl}
For $V\in\Z(H,\alpha|_H)$, the character $\chi_{E(V)}$ has the form 
\begin{equation}\label{char1}
\chi_{E(V)}\left(f,g\right)=\sum_{y\in Y}{\frac{\alpha\left(y^{-1}gy,y^{-1}|f\right)}{\alpha\left(y^{-1},g|f\right)}\chi_{V}\left(y^{-1}fy,y^{-1}gy\right)},\qquad f,g\in G,
\end{equation}
where $Y$ is a set of representatives in $G$ of the cosets from
$$\{y\in G|\ y^{-1}fy\in H,\ y^{-1}gy\in H\}/H\ \subset\ G/H\ .$$
\end{lem}
\begin{proof}
Let $V\in\Z\left(H,\alpha|_H\right)$. It follows from the defining condition that functions from $E(V)$ are supported by unions of right $H$-cosets of $G$. Clearly, any function from $E(V)$ is (a unique) sum of functions, each of which is supported on a single coset.  For a coset $yH$, a function $a\in E(V)$ with the support $\mathrm{supp}(a)=yH$ is uniquely determined by its value $a(y)=v\in V$. Denote such a function by $a_{y,v}$.  The space $E(V)$ is spanned by $a_{y,v}$ for $y\in G$ and $v\in V$. 
\nl
By \eqref{mapdeg} the degree of $a_{y,v}$ is $f$ iff $v\in V_{y^{-1}fy}$. 
Note that for $v\in V_{h^{-1}y^{-1}fyh}$ one has 
$$\alpha(h^{-1},y^{-1}|f)\alpha(h,h^{-1}|y^{-1}fy)a_{yh,v} = a_{y,h.v},\qquad h\in H\ .$$
Indeed, $v = a_{yh,v}(yh) = \alpha(h^{-1},y^{-1}|f)h^{-1}.a_{yh,v}(y)$ gives 
$$a_{y,h.v}(y) = h.v = \alpha(h^{-1},y^{-1}|f)\alpha(h,h^{-1}|y^{-1}fy)a_{yh,v}(y)\ .$$
Thus we can write 
\begin{equation}\lb{dec}
E(V)_f = \bigoplus_y\langle a_{y,v} |\ v\in V_{y^{-1}fy}\rangle \simeq \bigoplus_y{V_{y^{-1}fy}},
\end{equation}
where the sum is taken over representatives of the cosets $G/H$.

For $g\in C_{G}\left(f\right)$, consider the linear operator $g:E(V)_f\rightarrow E(V)_f$. Recall (from the proof of theorem \ref{loc1})
that the action of $g$ on $E(V)$ is given by 
$$(g.a)(x)=\alpha(x^{-1},g| f)^{-1}a(g^{-1}x),\qquad x\in G,\ a\in E(V)_f\ .$$
In particular $\mathrm{supp}(g.a)=g\ \mathrm{supp}(a)$. Thus $g.a_{y,v}$ is supported by the coset $gyH$ and hence can be written as $a_{gy,v'}$ for some $v'\in V$. 
We can say that the action of $g$ on $E(V)_f$ permutes the direct summands of \eqref{dec} according to the left action of $g$ on the cosets $G/H$.
\nl
Thus the trace ${\mathrm{Tr}_{E(V)_f}(g)}$ is a sum $\sum_y{\mathrm{Tr}_{\langle a_{y,v} |\ v\in V_{y^{-1}fy}\rangle}(g)}$, where $y$ runs through representatives of those cosets $yH$ for which $gyH=yH$. 
Let now $y\in G$ be such that $gyH=yH$ (or equivalently $y^{-1}gy\in H$). To compute the trace $\mathrm{Tr}_{\langle a_{y,v} |\ v\in V_{y^{-1}fy}\rangle}(g)$ note that $g.a_{y,v}$ (with $v\in V_{y^{-1}fy}$) can be written as $a_{y,v'}$ for some $v'\in V$. More explicitly 
$$v' = g.a_{y,v}(y) = \alpha(y^{-1},g| f)^{-1}a_{y,v}(g^{-1}y) = \alpha(y^{-1},g| f)^{-1}a_{y,v}(yy^{-1}g^{-1}y) = $$
$$= \frac{\alpha(y^{-1}gy,y^{-1}|f)}{\alpha(y^{-1},g| f)}(y^{-1}gy).a_{y,v}(y) = \frac{\alpha(y^{-1}gy,y^{-1}|f)}{\alpha(y^{-1},g| f)}(y^{-1}gy).v\ .$$
Thus the trace $\mathrm{Tr}_{\langle a_{y,v} |\ v\in V_{y^{-1}fy}\rangle}(g)$ coincides with $\frac{\alpha(y^{-1}gy,y^{-1}|f)}{\alpha(y^{-1},g| f)}\mathrm{Tr}_{V_{y^{-1}fy}}(y^{-1}gy)$.
\nl
Finally  $$\chi_{E(V)}\left(f,g\right)=\mathrm{Tr}_{E(V)_f}\left(g\right)=\sum_{y\in Y}\frac{\alpha(y^{-1}gy,y^{-1}|f)}{\alpha(y^{-1},g| f)}\mathrm{Tr}_{V_{y^{-1}fy}}(y^{-1}gy) = $$ $$= \sum_{y\in Y}{\frac{\alpha\left(y^{-1}gy,y^{-1}|f\right)}{\alpha\left(y^{-1},g|f\right)}\chi_{V}\left(y^{-1}fy,y^{-1}gy\right)}\ .$$ 
\end{proof}

\begin{lem}\label{tgac}
Let $\gamma\in C^2(H,k^*)$ be a coboundary for $\alpha|_H$. The $\mathcal{Z}(H,\alpha|_H)$ character $\chi_{k[H,\gamma]}$ of the twisted group algebra $k[H,\gamma]$ is given by 
$$
\chi_{k[H,\gamma]}(h,u) = \frac{\gamma(u,h)}{\gamma(h,u)}\ .
$$
\end{lem}
\begin{proof}
For any $h\in H$, the degree-$h$ component $k\left[H,\gamma\right]_h$ is one-dimensional and has the form $$k\left[H,\gamma\right]_h=ke_{h}\ .$$ According to \eqref{atga} $$u.e_{h}=\frac{\gamma\left(u,h\right)}{\gamma\left(uhu^{-1},u\right)}e_{uhu^{-1}},\qquad u\in H\ .$$  For $u\in C_{H}\left(h\right)$, this becomes
$$
u.e_{h}=\frac{\gamma\left(u,h\right)}{\gamma\left(h,u\right)}e_{h}\ .
$$
\end{proof}

Now we are ready to prove the main result of this section.
\bth\lb{chla}
Let $L=L\left(H,\gamma\right)$ be a Lagrangian algebra in  $\Z(G,\alpha)$.  Then the character $\chi_{L}$ has the form
\begin{equation}\label{char4}
\chi_{L}\left(f,g\right)=\sum_{y\in Y}{\frac{\alpha\left(y^{-1}gy,y^{-1}|
f\right)\gamma\left(y^{-1}fy,y^{-1}gy\right)}{\alpha\left(y^{-1},g|f\right)\gamma\left(y^{_1}gy,y^{-1}fy\right)}},\qquad f,g\in G,
\end{equation}
where $Y$ is a set of representatives in $G$ of the cosets from
$$\{y\in G|\ y^{-1}fy\in H,\ y^{-1}gy\in H\}/H\ \subset\ G/H\ .$$
\eth
\begin{proof}
Follows from lemmas \ref{charl} and \ref{tgac}.
\end{proof}

\subsection{Example: $\Z(D_3,\alpha)$}\lb{s3e}

Here $D_3$ is the dihedral group of order 6: $D_3=\langle r,s|\ r^3=s^2=e;srs=r^{-1}\ \rangle.$
\nl
It is known that the third cohomology group $H^3(D_3,\mathbb{C}^*)$ is cyclic of order $6$ (see e.g. \cite[section 6.3]{cg}).  More explicitly, the cohomology group $H^{3}\left(D_3,\mathbb{C}^*\right)$ is generated by the class of the 3-cocycle $\theta$ given by
$$\theta\left(s^{m_1}r^{n_1},s^{m_2}r^{n_2},s^{m_3}r^{n_3}\right)=$$ $$=\exp\left(\frac{2\pi i}{9}\left((-1)^{m_2+m_3}n_{1}\Big((-1)^{m_3}n_2+n_{3}-\left[(-1)^{m_3}n_2+n_3\right]_{3}\Big)+\frac{9}{2}m_1m_2m_3\right)\right)\ ,$$
where $[\ \ ]_3$ denotes taking the residue modulo 3.  

Recall that the second cohomology group $H^2(D_3,\mathbb{C}^*)$ is trivial. This implies that a coboundary for the restriction $\alpha|_H$ of a 3-cocycle $\alpha\in Z^3(D_3,\mathbb{C}^*)$ to a subgroup $H\subset G$ is unique if exists. Thus Lagrangian algebras in $\Z(D_3,\alpha)$ are labelled just by (conjugacy classes of) subgroups on which the restriction of $\alpha$ is cohomologically trivial ({\em admissible subgroups}).

In terms of the cocycle $\alpha$, there are four distinct cases, depending on the order of the class of $\alpha$ in $H^3(D_3,\mathbb{C}^*)$.
\nl
{\bf The case of trivial $\alpha$}
\nl
In this case $\Z(D_3,\alpha) = \Z(D_3)$.  The character table for $\Z(D_3)$ is
$$
\begin{array}{c|cccccccc}
  & (e,e) & \left(e,r\right) & \left(e,s\right) & \left(r,e\right) & \left(r,r\right) & \left(r,r^2\right) & \left(s,e\right) & \left(s,s\right)\\
\hline
\chi_0 & 1 & 1 & 1 & 0 & 0 & 0 & 0 & 0\\
\chi_1 & 1 & 1 & -1 & 0 & 0 & 0 & 0 & 0\\
\chi_2 & 2 & -1 & 0 & 0 & 0 & 0 & 0 & 0\\
\chi_3 & 0 & 0 & 0 & 1 & 1 & 1 & 0 & 0\\
\chi_4 & 0 & 0 & 0 & 1 & \omega & \omega^{-1} & 0 & 0\\
\chi_5 & 0 & 0 & 0 & 1 & \omega^{-1} & \omega & 0 & 0\\
\chi_6 & 0 & 0 & 0 & 0 & 0 & 0 & 1 & 1\\
\chi_7 & 0 & 0 & 0 & 0 & 0 & 0 & 1 & -1\\
\end{array}$$
Here, $\omega$ is a primitive third root of unity.
\nl
All subgroups are admissible. Up to conjugation there are four subgroups
$$\{e\},\qquad C_2=\langle s\rangle,\qquad C_3=\langle r\rangle,\qquad D_3\ .$$
The characters of the corresponding Lagrangian algebras are
$$
\begin{array}{c|cccccccc}
  & (e,e) & \left(e,r\right) & \left(e,s\right) & \left(r,e\right) & \left(r,r\right) & \left(r,r^2\right) & \left(s,e\right) & \left(s,s\right)\\
\hline
\chi_{L(\{e\})} & 4 & 0 & 0 & 0 & 0 & 0 & 0 & 0\\
\chi_{L(C_2)} & 3 & 0 & 1 & 0 & 0 & 0 & 1 & 1\\
\chi_{L(C_3)} & 2 & 2 & 0 & 2 & 2 & 2 & 0 & 0\\
\chi_{L(D_3)} & 1 & 1 & 1 & 1 & 1 & 1 & 1 & 1\\
\end{array}$$
which give the decompositions into irreducible characters
\begin{align*}
\chi_{L(\{e\})} & = \chi_0+\chi_1+2\chi_2\  ,\\
\chi_{L(C_2)} & =  \chi_0+\chi_2+\chi_6\ ,\\
\chi_{L(C_3)} & =  \chi_0+\chi_1+2\chi_3\ ,\\
\chi_{L(D_3)} & = \chi_0+\chi_3+\chi_6 \ .
\end{align*}
\nl
{\bf The case of $\alpha$ of order two.}
\nl
In this case, $\alpha = \theta^3$.  The character table for $\Z(D_3,\theta^3)$ is
$$
\begin{array}{c|cccccccc}
 & (e,e) & \left(e,r\right) & \left(e,s\right) & \left(r,e\right) & \left(r,r\right) & \left(r,r^2\right) & \left(s,e\right) & \left(s,s\right)\\
\hline
\chi_0 & 1 & 1 & 1 & 0 & 0 & 0 & 0 & 0\\
\chi_1 & 1 & 1 & -1 & 0 & 0 & 0 & 0 & 0\\
\chi_2 & 2 & -1 & 0 & 0 & 0 & 0 & 0 & 0\\
\chi_3 & 0 & 0 & 0 & 1 & 1 & 1 & 0 & 0\\
\chi_4 & 0 & 0 & 0 & 1 & \omega & \omega^{-1} & 0 & 0\\
\chi_5 & 0 & 0 & 0 & 1 & \omega^{-1} & \omega & 0 & 0\\
\chi_6 & 0 & 0 & 0 & 0 & 0 & 0 & 1 & \varepsilon\\
\chi_7 & 0 & 0 & 0 & 0 & 0 & 0 & 1 & -\varepsilon\\
\end{array}$$
Here, $\omega$ and $\varepsilon$ are primitive third and fourth roots of unity, respectively.
\nl 
The admissible subgroups are $\{e\}$ and $C_3$.  The characters of the corresponding Lagrangian algebras are
$$
\begin{array}{c|cccccccc}
  & (e,e) & \left(e,r\right) & \left(e,s\right) & \left(r,e\right) & \left(r,r\right) & \left(r,r^2\right) & \left(s,e\right) & \left(s,s\right)\\
\hline
\chi_{L(\{e\})} & 4 & 0 & 0 & 0 & 0 & 0 & 0 & 0\\
\chi_{L(C_3)} & 2 & 2 & 0 & 2 & 1+\omega^{-1} & 1+\omega & 0 & 0\\
\end{array}$$
which give the decompositions into irreducible characters
\begin{align*}
\chi_{L(\{e\})} & = \chi_0+\chi_1+2\chi_2\  ,\\
\chi_{L(C_3)} & =  \chi_0+\chi_1+\chi_3+\chi_5\ ,\\
\end{align*}
\nl
{\bf The case of $\alpha$ of order three.}
\nl
In this case, $\alpha = \theta^2$.  The character table for $\Z(D_3,\theta^2)$ is
$$
\begin{array}{c|cccccccc}
 & (e,e) & \left(e,r\right) & \left(e,s\right) & \left(r,e\right) & \left(r,r\right) & \left(r,r^2\right) & \left(s,e\right) & \left(s,s\right)\\
\hline
\chi_0 & 1 & 1 & 1 & 0 & 0 & 0 & 0 & 0\\
\chi_1 & 1 & 1 & -1 & 0 & 0 & 0 & 0 & 0\\
\chi_2 & 2 & -1 & 0 & 0 & 0 & 0 & 0 & 0\\
\chi_3 & 0 & 0 & 0 & 1 & \eta^{1} & \eta^{2} & 0 & 0\\
\chi_4 & 0 & 0 & 0 & 1 & \eta^{4} & \eta^{8} & 0 & 0\\
\chi_5 & 0 & 0 & 0 & 1 & \eta^{7} & \eta^{3} & 0 & 0\\
\chi_6 & 0 & 0 & 0 & 0 & 0 & 0 & 1 & 1\\
\chi_7 & 0 & 0 & 0 & 0 & 0 & 0 & 1 & -1\\
\end{array}$$
Here $\eta$ is a primitive ninth root of unity.
The admissible subgroups are $\{e\}$ and $C_2$. 
The characters of the corresponding Lagrangian algebras are
$$
\begin{array}{c|cccccccc}
  & (e,e) & \left(e,r\right) & \left(e,s\right) & \left(r,e\right) & \left(r,r\right) & \left(r,r^2\right) & \left(s,e\right) & \left(s,s\right)\\
\hline
\chi_{L(\{e\})} & 4 & 0 & 0 & 0 & 0 & 0 & 0 & 0\\
\chi_{L(C_2)} & 3 & 0 & 1 & 0 & 0 & 0 & 1 & 1\\
\end{array}$$
which give the decompositions into irreducible characters
\begin{align*}
\chi_{L(\{e\})} & = \chi_0+\chi_1+2\chi_2\  ,\\
\chi_{L(C_2)} & =  \chi_0+\chi_2+\chi_6\ ,\\
\end{align*}
\nl
{\bf The case of $\alpha$ of order six.}
\nl
In this case, $\alpha = \theta$.  The character table for $\Z(D_3,\theta)$ is
$$
\begin{array}{c|cccccccc}
\Z(D_3,\theta) & (e,e) & \left(e,r\right) & \left(e,s\right) & \left(r,e\right) & \left(r,r\right) & \left(r,r^2\right) & \left(s,e\right) & \left(s,s\right)\\
\hline
\chi_0 & 1 & 1 & 1 & 0 & 0 & 0 & 0 & 0\\
\chi_1 & 1 & 1 & -1 & 0 & 0 & 0 & 0 & 0\\
\chi_2 & 2 & -1 & 0 & 0 & 0 & 0 & 0 & 0\\
\chi_3 & 0 & 0 & 0 & 1 & \eta & \eta^2 & 0 & 0\\
\chi_4 & 0 & 0 & 0 & 1 & \eta^4 & \eta^8 & 0 & 0\\
\chi_5 & 0 & 0 & 0 & 1 & \eta^7 & \eta^3 & 0 & 0\\
\chi_6 & 0 & 0 & 0 & 0 & 0 & 0 & 1 & \varepsilon\\
\chi_7 & 0 & 0 & 0 & 0 & 0 & 0 & 1 & -\varepsilon\\
\end{array}$$
Here again, $\varepsilon$ and $\eta$ are primitive fourth and ninth roots of unity, respectively.
\nl
Only the trivial subgroup $\{e\}$ is admissible.  The character of the corresponding Lagrangian algebra is 
$$
\begin{array}{c|cccccccc}
  & (e,e) & \left(e,r\right) & \left(e,s\right) & \left(r,e\right) & \left(r,r\right) & \left(r,r^2\right) & \left(s,e\right) & \left(s,s\right)\\
\hline
\chi_{L(\{e\})} & 4 & 0 & 0 & 0 & 0 & 0 & 0 & 0\\
\end{array}$$
with the decomposition into irreducible characters
$$\chi_{L(\{e\})}  = \chi_0+\chi_1+2\chi_2\  .$$

\section*{Appendix A.Certain monoidal categories associated with finite groups and Hochschild-Serre spectral sequence.}\lb{appx}

Let $G$ be a finite group and let $F$ be another finite group, acting on $G$ by group automorphisms. Let
$\gamma:F\times F\times G\to k^*$ be a fuction, satisfying
\begin{equation}\label{1}
\gamma(f,gh|x)\gamma(g,h|x) = \gamma(fg,h|x)\gamma(f,g|h(x)),\quad f,g,h\in F,\ x\in G
\end{equation}
and the normalisation condition: $$\gamma(e,g|x) = \gamma(f,e|x) = \gamma(f,g|e) = 1.$$ Define the category
$\cC(F,G,\gamma)$, whose objects are (finite-dimensional) $G$-graded vector spaces $V=\oplus_{x\in G}V_x$, equipped
with $F$-action $$f:V\to V,\ f(V_x) = V_{f(x)},$$ which is $\gamma$-projective $$(fg)(v) = \gamma(f,g|x)f(g(v)),\quad
x\in V_x.$$ Morphisms are graded, action preserving maps.

Now let $\beta:F\times G\times G\to k^*$ be a normalised function, satisfying
\begin{equation}\label{2}
\beta(fg|x,y)\gamma(f,g|x)\gamma(f,g|y) = \gamma(f,g|xy)\beta(g|x,y)\beta(f|g(x),g(y)).
\end{equation}
Define a tensor product in the category $\cC(F,G,\gamma)$ by $$(U\otimes V)_z = \oplus_{xy=z}U_x\otimes V_y,$$ with the
$F$-action $$f(u\otimes v) = \beta(f|x,y)f(u)\otimes f(v).$$ The condition (\ref{2}) implies that this action is
$\gamma$-projective.

Finally, a normalised function $\alpha:G\times G\times G\to k^*$, satisfying
\begin{equation}\label{3}
\alpha(x,y,z)\beta(f|xy,z)\beta(f|x,y) = \beta(f|x,yz)\beta(f|y,z)\alpha(f(x),f(y),f(z))
\end{equation}
and a 3-cocycle condition
\begin{equation}\label{4}
\alpha(y,z,w)\alpha(x,yz,w)\alpha(x,y,z) = \alpha(x,y,zw)\alpha(xy,z,w).
\end{equation}
Then the map $\alpha:U\otimes (V\otimes W)\to (U\otimes V)\otimes W$
$$u\otimes(v\otimes w)\mapsto\alpha(x,y,z)(u\otimes v)\otimes w,\quad u\in U_x, v\in V_y, w\in W_z$$
is a morphism in the category $\cC(F,G,\gamma)$ (with condition (\ref{3} implying its $F$-equivariance), satisfying the
pentagon axiom (by condition (\ref{4})).
\newline
We denote by $\cC(F,G,\gamma,\beta,\alpha)$ the category $\cC(F,G,\gamma)$ with the monoidal structure defined by
$\beta$ and $\alpha$.

Representation categories and categories of graded vector spaces fit in the series of categories
$\cC(F,G,\gamma,\beta,\alpha)$ as extreme cases, with (categories of modules in) group-theoretical modular categories
appearing as intermediate examples. Indeed, for $G=\{e\}$, $\gamma=\beta=\alpha=1$, the category
$\cC(F,G,\gamma,\beta,\alpha)$ is the category $\Rep(G)$ of representations of $G$. If $F=\{e\}$, $\gamma=\beta=1$,
$\cC(F,G,\gamma,\beta,\alpha)$ is the monoidal category $\cC(G,\alpha)$ of $G$-graded vector spaces with the associator
given by $\alpha$. For $F=G$ with conjugation action and with $\gamma$ and $\beta$ defined by (\ref{acfa}) and
(\ref{tefa}), the category $\cC(F,G,\gamma,\beta,\alpha)$ coincides with $\Z(G,\alpha)$. Finally, for $H$ being a
subgroup of $G$ again with conjugation action and with $\gamma$ and $\beta$ defined by (\ref{acfa}) and (\ref{tefa}),
the category $\cC(H,G,\gamma,\beta,\alpha)$ coincides with the category $\Z(G,\alpha)_{k(G/H)}$ of modules of the
algebra $k(G/H)$ in $\Z(G,\alpha)$.

To compare different categories of the form $\cC(F,G,\gamma,\beta,\alpha)$ we introduce certain monoidal equivalences
between them, which we call for short {\em gauge} equivalences. Let $\cC(F,G,\gamma,\beta,\alpha)$ and
$\cC(F,G,\gamma',\beta',\alpha')$ be two (monoidal) categories. Let $b:F\times G\to k^*$ be a normalised fuction,
satisfying
\begin{equation}\label{c1}
b(fg|x)\gamma(f,g|x) = \gamma'(f,g|x)b(f|x)b(f|g(x)).
\end{equation}
Define a functor $T(b):\cC(F,G,\gamma)\to\cC(F,G,\gamma')$, which preserves the $G$-grading on an object $V$ and
modifies the $F$-action: $$\tilde f(v) = b(f|x)f(v),\quad v\in V_x.$$ The condition (\ref{c1}) implies that the new
action is $\gamma'$-projective. Now let $a:G\times G\to k^*$ be a normalised function such that
\begin{equation}\label{c2}
b(f|xy)a(x,y)\beta(f|x,y) = \beta'(f|x,y)b(f|x)b(f|y)a(f(x),f(y)),
\end{equation}
\begin{equation}\label{c3}
\alpha(x,y,z)a(x,y)a(xy,z) = a(y,z)a(x,yz)\alpha'(x,y,z).
\end{equation}
Define a monoidal structure on the functor $T(b)$: $$u\otimes v\mapsto b(x,y)u\otimes v,\quad u\in U_x, v\in V_y.$$ The
condition (\ref{c2}) implies that this map is $F$-equivariant, while (\ref{c3}) is equivalent to the coherence axiom
for a monoidal functor. Denote the resulting monoidal equivalence by $T(b,a):\cC(F,G,\gamma,\beta,\alpha)\to
\cC(F,G,\gamma',\beta',\alpha')$.

To describe gauge equivalence classes of categories $\cC(F,G,\gamma,\beta,\alpha)$ consider a double complex
$C^{p,q}(F,G,k^*)=C^p(F,C^q(G,k^*))$. Elements of $C^{p,q}(F,G,k^*)$ are normalised functions $$c:F^{\times p}\times
G^{\times q}\to k^*,\quad (f_1,...,f_p,g_1,...,g_q)\mapsto c(f_1,...,f_p|g_1,...,g_q).$$ The vertical differential
$\partial:C^{p,q}\to C^{p,q+1}$ is induced by the standard differential $C^q(G,k^*)\to C^{q+1}(G,k^*)$:
$$\partial(c)(f_1,...,f_p|g_1,...,g_{q+1}) =$$
$$c(f_1,...,f_p|g_2,...,g_{q+1})(\prod_{i=1}^qc(f_1,...,f_p|g_1,...,g_ig_{i+1},...,g_q)^{(-1)^i})c(f_1,...,f_p|g_1,...,g_q)^{(-1)^{q+1}}.$$
The horizontal differential $d:C^{p,q}\to C^{p,q+1}$ is the standarad differential itself:
$$d(c)(f_1,...,f_{p+1}|g_1,...,g_q) =$$
$$c(f_2,...,f_{p+1}|g_1,...,g_q)(\prod_{i=1}^pc(f_1,...,,f_if_{i+1},...,f_p|g_1,...g_q)^{(-1)^i})c(f_1,...,f_p|f_p(g_1),...,f_p(g_q))^{(-1)^{p+1}}.$$
The conditions (\ref{1})-(\ref{4}) can be rewritten as
$$d(\gamma)=1,\quad \partial(\gamma)=d(\beta),\quad \partial(\beta)=d(\alpha),\quad \partial(\alpha)=1.$$
In other words, the collection $(\gamma,\beta,\alpha)$ is a 3-cocycle of the truncated total complex $\tilde
C^n(F,G,k^*) = \oplus_{p=0}^{n-1}C^{p,n-p}(F,G,k^*)$. The conditions (\ref{c1})-\ref{c3}) say that
$$\gamma'd(b)=\gamma,\qquad \beta'd(a)=\beta\partial(b),\qquad \alpha'\partial(a)=\alpha.$$
Hence $(b,c)$ is a 2-coboundary in $\tilde C^*$ for the collection
$(\gamma,\beta,\alpha)(\gamma',\beta',\alpha')^{-1}$. Thus gauge equivalence classes of categories of the form
$\V(F,G,\gamma,\beta,\alpha)$ correspond to 3-cohomology classes of $\tilde C^*$. Note that $\tilde C^*$ is a direct
summand of the (untruncated) total complex $C^n=\oplus_{p=0}^nC^{p,n-p}$, which is quasi-isomorphic to the standard
complex $C^*(F\ltimes G,k^*)$ of the crossed product of groups $F\ltimes G$ \cite{hs}. Since the kernel of the
projection  $C^*\to \tilde C^*$ coincides with $C^*(F,k^*)$ we have an isomorphism
$$H^n(\tilde C)\oplus H^n(F,k^*)\simeq H^n(F\ltimes G,k^*).$$
In particular, when $F=G$, acting on itself by conjugation, the crossed product of groups $G\ltimes G$ is isomorphic to
the product $G\times G$. In this case the cohomology of $C^*$ and $\tilde C^*$ are $$H^n(C)=\oplus_{p=0}^n
H^p(G,H^{n-p}(G,k^*)),\quad H^n(\tilde C)=\oplus_{p=0}^{n-1}H^p(G,H^{n-p}(G,k^*)).$$ A cochain map
$\tau:C^n(G,k^*)=C^0(G,C^n(G,k^*))\to\tilde C^n$ giving the splitting in cohomology $H^n(G,k^*)=H^0(G,H^n(G,k^*))\to
H^n(\tilde C)$ was constructed in \cite{hs} and has the following form: $$\tau(\alpha) = \sum_{p=0}^{n-1}\alpha_p,$$
where $\alpha_p\in C^p(G,C^{n-p}(G,k^*))$ is $$\alpha_p(f_1,...,f_p|g_1,...,g_{n-p}) =
\prod_{\sigma}\alpha(x_1,...,x_n)^{(-1)^\sigma}.$$ Here the product is over all $(p,n-p)$-shuffles and $x_i =
h_ig_{\sigma(i)}h_i^{-1}$ with $h_i$ being the ordered product $\prod_{\sigma(j)>\sigma(i)}f_j$. For example, for
$\alpha\in C^3(G,k^*)$
$$\alpha_1(f|g,h)=\alpha(f,g,h)\alpha(fgf^{-1},f,h)^{-1}\alpha(fgf^{-1},fhf^{-1},f)=\alpha(f|g,h),$$
$$\alpha_2(f,g|h)=\alpha(f,g,h)\alpha(f,ghg^{-1},g)^{-1}\alpha(fghg^{-1}f^{-1},ghg^{-1},g)=\alpha(f,g|h).$$
Thus the monoidal category $\Z(G,\alpha)$ coincides with $\cC(F,G,\gamma,\beta,\alpha)$, where $(\gamma,\beta,\alpha)=\tau(\alpha)$.

Now we return to the categories $\cC(F,G,\gamma',\beta',\alpha')$. We assume that $F=G$ with the conjugation action on $G$.
In this case we can ask when the functorial isomorphism, defined by
\begin{equation}\label{bra}
c_{U,V}(u\otimes v) = c(x,y)x(v)\otimes u,\quad u\in U_x, v\in V_y.
\end{equation}
is a braiding in $\cC(F,G,\gamma',\beta',\alpha')$.
The hexagon axioms are equivalent to the equations
\begin{equation}\label{b1}
\beta(x|y,z)c(x,yz) = c(x,y)c(x,z)\alpha(x|y,z),
\end{equation}
\begin{equation}\label{b2}
c(xy,z)\gamma(z|x,y) = c(y,z)c(x,yzy^{-1})\alpha(x,y|z).
\end{equation}
Denote by $\cC(G,G,\gamma,\beta,\alpha,c)$ a braided category $\cC(G,G,\gamma,\beta,\alpha)$ with the braiding
(\ref{bra}). A gauge equivalence $T(b,a):\cC(F,G,\gamma,\beta,\alpha,c)\to \cC(F,G,\gamma',\beta',\alpha',c')$ is
braided iff
\begin{equation}\label{beq}
b(x,y)c'(x,y)a(x|y) = b(xyx^{-1},x)c(x,y).
\end{equation}

The equations (\ref{b1}),(\ref{b2}) can be intrerpreted as a coboundary condition
$d(c)=(\gamma,\beta,\alpha)\tau(\alpha)^{-1}$ (here $c$ is considered as an element of $C^1(G,C^1(G,k^*))$). Thus, at
least at the level of monoidal categories, any $\cC(G,G,\gamma,\beta,\alpha)$ with braiding of the form (\ref{bra}) is
equivalent to $\Z(G,\alpha)$. Note that in the assumption $(\gamma,\beta,\alpha)=\tau(\alpha)$ the equations
(\ref{b1}),(\ref{b2}) turn into bi-multiplicativity condition for $c$: $$c(x,yz)=c(x,y)c(x,z),\quad
c(xy,z)=c(x,z)c(y,z).$$ Thus by (\ref{beq}) the guage equivalence $T(1,c)$ is a braided equivalence between
$\cC(G,G,\gamma,\beta,\alpha,c)$ (with $(\gamma,\beta,\alpha)=\tau(\alpha)$) and $\Z(G,\alpha)$. In particular we have
proved the following.
\begin{prop}\label{chace}
Braided monoidal category of the form $\cC(G,G,\gamma,\beta,\alpha,c)$ is braided equivalent to $\Z(G,\alpha)$.
\end{prop}


\begin{thebibliography}{11}
\bibitem{ba1}
P. B\'antay, Algebraic Aspects of Orbifold Models,  Int. J. of Modern Physics A9 (1994), 1443.
%
\bibitem{cg}
A. Coste, T. Gannon, P. Ruelle,
Finite group modular data. Nuclear Phys. B 581 (2000), no. 3, 679--717.
%
\bibitem{da3}
A. Davydov, Modular invariants for group-theoretical modular data I, \textit{J. Algebra}, 323 (2010), pp. 1321-1348, arXiv:0908.1044
%
\bibitem{da4} 
A. Davydov, Centre of an algebra, \textit{Advances in Mathematics}, 225 (2010), 319�348, arXiv:0908.1250
%
\bibitem{dmno} 
A. Davydov, M. M\"uger, D. Nikshych, V. Ostrik, Witt group of non-degenerate braided fusion categories,   \textit{Journal f\"ur die Reine und Angewandte Mathematik}, (2013), 677, 135-177, arXiv:1009.2117. 
%
\bibitem{dgno}
V. Drinfeld, S. Gelaki, D. Nikshych, and V. Ostrik, On braided fusion categories I, 
Selecta Math. 16 (2010), 1-119. 
%
\bibitem{eo}
P. Etingof, V. Ostrik, Finite tensor categories, Mosc. Math.~J.  (2004),  4, no. 3, 627--654
%
\bibitem{ffrs}
J. Fr\"ohlich, J. Fuchs, I. Runkel, C. Schweigert,
Correspondences of ribbon categories. Adv. Math. 199 (2006), no. 1, 192--329.
%
\bibitem{hs}
G. Hochschild, J.-P. Serre,
Cohomology of group extensions.
Trans. Amer. Math. Soc. 74, (1953). 110--134.
%
\bibitem{ko}
A. Kirillov Jr., V. Ostrik,
On a $q$-analogue of the McKay correspondence and the ADE classification of $sl_2$ conformal field theories.
Adv. Math. 171 (2002), no. 2, 183--227.
%
\bibitem{os}
V. Ostrik, Module categories, weak Hopf algebras and modular invariants. Transform. Groups 8 (2003), no. 2, 177--206.
%
\bibitem{os1}
V. Ostrik,
Module categories over the Drinfeld double of a finite group. Int. Math. Res. Not. 2003, no. 27, 1507--1520.
%
\bibitem{pa}
B. Pareigis,
On braiding and dyslexia.
J. Algebra 171 (1995), no. 2, 413--425.
%
\bibitem{tu}
V. Turaev,
Homotopy field theory in dimension 2 and group-algebras.
arXiv:math/9910010
%

\end{thebibliography}
\end{document}